\documentclass[11pt]{article}
 \usepackage{epsfig}
 \usepackage{amssymb,amsmath,amsthm,amscd}
\usepackage{latexsym}
\usepackage{graphicx}
\usepackage{color}
\usepackage{mdwlist}
\usepackage{subfigure}
\usepackage{array}
\usepackage{bigints}
\usepackage{pgfplots}
\usepackage{hyperref}
 \newtheorem{thm}{Theorem}[section]
 \newtheorem{lem}[thm]{Lemma}
  \newtheorem{rem}[thm]{Remark}
 
 \newtheorem{prop}[thm]{Proposition}
 \newtheorem{cor}[thm]{Corollary}

\def\curl{\mbox{curl\,}}
 \newcommand{\n}{{\bf n}}
\newcommand{\bu}{{\bf u}}

\newcommand{\R}{\mathbb {R}}

\newcommand{\Z}{\mathbb {Z}}
 \title{Bifurcation and stability of stationary shear flows of Ericksen-Leslie model for nematic liquid crystals}

\usepackage{authblk}
\author{Weishi Liu\thanks{Department of Mathematics, University of Kansas. Email: \texttt{wsliu@ku.edu}}\,\, \& Majed Sofiani \thanks{Computer, Electrical, Mathematical Sciences \& Engineering Division, King Abdullah University of Science and
Technology (KAUST). Email: \texttt{majed.sofiani@kaust.edu.sa}}}
{\tiny}

\begin{document}
 \maketitle

\begin{abstract}  In this work, focusing on a critical case for shear flows of nematic liquid crystals, we investigate multiplicity and stability of stationary solutions via the parabolic Ericksen-Leslie system. { We establish a one-to-one correspondence between the set of the stationary solutions with the set of the solutions of {\em an algebraic equation} for a cusp case. This one-to-one correspondence is established essentially based on the treatment in the work of Jiao, et. al. [{\em J. Diff. Dyn. Syst. {\bf 34} (2022), 239-269}] for a different case, and the relation  gives directly parameter ranges for existence of multiple stationary solutions; in particular,  multiple stationary solutions are created through countably many saddle-node bifurcations for {\em the algebraic equation} at critical shear speeds.  The main result of the paper is on the stability of stationary solutions associated to the bifurcations;} more precisely, (i) for each critical shear speed,  there is a unique stationary solution and, for smaller shear speed,  the stationary solution disappears but, for larger shear speed,  two stationary solutions nearby bifurcate;  (ii) more importantly, under a generic condition,  there is a simple zero eigenvalue for the linearization of the shear flow at the critical stationary solution and, for larger shear speed, the zero eigenvalue bifurcates to a negative eigenvalue for one of the two stationary solutions  and to a  positive eigenvalue for the other stationary solution.   

\bigskip

\end{abstract}
{\large}

\newpage
\tableofcontents

\newpage

 \section{Introduction}\label{IntroSect}
 \setcounter{equation}{0}

\hspace{\parindent} Liquid crystals represent a distinct state of matter that combines fluidity with anisotropic properties typically associated with crystalline solids. Among their various forms, the \emph{nematic} phase is characterized by molecules that are effectively rod-like or thread-like. At the macroscopic level, a nematic liquid crystal is described by a \textit{velocity field} \( \mathbf{u} \in \mathbb{R}^3 \) governing fluid motion and a \textit{director field} \( \mathbf{n} \in \mathbb{S}^2 \) representing the average molecular alignment. These fields are strongly coupled: distortions in \( \mathbf{n} \) can induce flow \( \mathbf{u} \), and conversely, fluid motion affects molecular orientation. 

 The development of modeling and theory for nematic liquid crystals evolved over several decades. In the 1920s, Oseen \cite{oseen33} and Zocher \cite{Zocher1927} proposed an initial model describing the elastic energy associated with deformations of the director field \( \mathbf{n} \), capturing the equilibrium response to splay, twist, and bend distortions. 
 
In the early 1960s, Ericksen \cite{Eri76} extended this framework by developing a variational theory that treated nematic liquid crystals as continuum materials with internal orientational structure. 
Soon after, Ericksen proposed a dynamic extension \cite{ericksen1961}, incorporating conservation laws of mass, momentum, and angular momentum into a framework suitable for time-dependent behavior, though the model lacked full constitutive specifications.
Leslie \cite{leslie1968some,les} completed the dynamic theory in the late 1960s by introducing a full hydrodynamic model. He derived constitutive equations for the stress tensor, including anisotropic viscous terms, and introduced six viscosity coefficients that characterize the material’s response to flow and director reorientation. The resulting system coupled the momentum equation for \( \mathbf{u} \) (a Navier--Stokes-type equation with anisotropic stress) with an evolution equation for \( \mathbf{n} \), capturing elasticity, flow-alignment interactions, and non-Newtonian effects. This formed a comprehensive continuum model for the dynamics of nematic liquid crystals.

In the rest of the introduction,  we discuss a brief overview of the model and its key characteristic quantities, and then introduce the initial-boundary value problem for shear flows that will be considered in this work and provide an informal summary of the main results.

\subsection{Ericksen-Leslie model for nematics}
Using the convention to denote $\dot{f}=f_t+\bu\cdot \nabla f$ the material derivative, the full Ericksen-Leslie (EL) model for nematics  is given as follows:
\begin{equation}\label{wlce}
\begin{cases}
\rho\dot \bu+\nabla P=\nabla\cdot\sigma-\nabla\cdot\left(\frac{\partial W}{\partial\nabla \n}\otimes\nabla \n\right),
\\
\nabla\cdot \bu=0,\\ 
\nu{\ddot \n}=\lambda\n-\frac{\partial W}{\partial  \n}-{\bf g}+\nabla\cdot\left(\frac{\partial W}{\partial\nabla \n}\right), \\
|\n|=1.\\
\end{cases}
\end{equation}
In  (\ref{wlce}),  $P$ is the pressure, $\lambda$ is the Lagrangian multiplier of the constraint $|\n|=1$, $\rho$ is the density, $\nu$ is the inertial coefficient of the director $\n$, and as we will describe below, $W, {\bf g}$  and $\sigma$ are the Oseen-Frank energy,  the kinematic transport and the viscous stress tensor, respectively. Here we give a brief overview of the model and the quantities involved (see, e.g., \cite{liucalderer00, frank58, leslie68,DeGP,ericksen62,les, lin89}
for details).\\
\begin{figure}[h]
    \centering
    \subfigure[]{\includegraphics[width=0.24\textwidth]{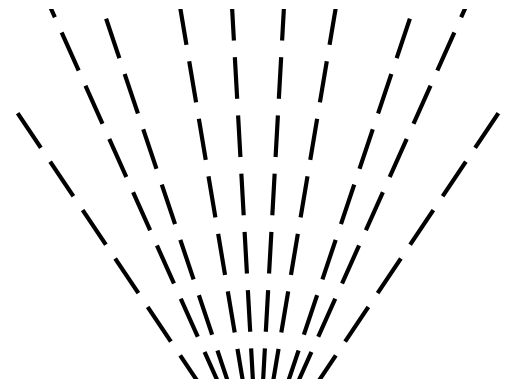}} 
    \subfigure[]{\includegraphics[width=0.24\textwidth]{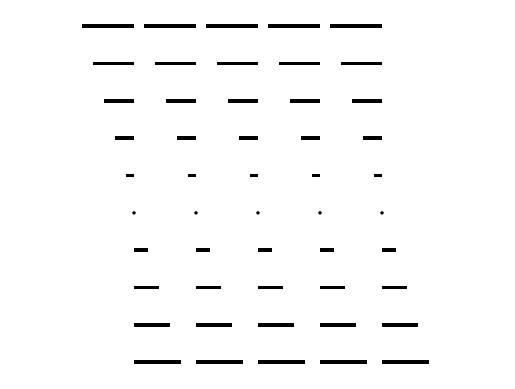}} 
    \subfigure[]{\includegraphics[width=0.24\textwidth]{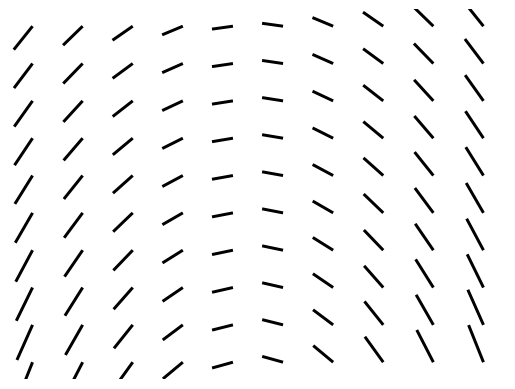}}
    \caption{(a) Spaly: $\nabla\cdot \n\neq 0$ (b) Twist: $\n\cdot\curl\n\neq 0$ (c) Bend: $|\n\times\curl\n|\neq 0$}
    \label{fig:Deform}
\end{figure}

At the level of the molecular structure, the  Oseen-Frank energy density $W$ determines the microstructure of the crystals 
(\cite{oseen33,frank58})
\begin{align}\label{OFE}\begin{split}
2W(\n,\nabla \n)=&K_1(\nabla\cdot \n)^2+K_2(\n\cdot\nabla\times \n)^2+K_3|\n\times(\nabla\times \n)|^2\\
&+(K_2+K_4)[\mbox{tr}(\nabla \n)^2-(\nabla\cdot \n)^2] 
\end{split}
\end{align}
where, as shown in Figure \ref{fig:Deform},  $K_j>0$, $j=1,2,3$, are the coefficients for {\em splay, twist, and bend} effects respectively, with
$K_2\geq |K_4|$, $2K_1\geq K_2+K_4$. This  Oseen-Frank energy density $W(\n,\nabla \n)$ should be viewed as the energy resulting from deforming the material with the above distortions.

The main and the most interesting feature of liquid crystals is the interaction between the molecules and hydrodynamic properties. This interaction is captured in the system above through the following quantities:

First, we introduce the two tensors $D$ and $\omega$ and the vector $N$ 
$$
D= \frac12(\nabla \bu+\nabla^{T} \bu),\quad \omega= \frac12(\nabla \bu-\nabla^{T}\bu),\quad N=\dot \n-\omega \n 
$$
representing the rate of strain tensor, skew-symmetric part of the strain rate, and the
rigid rotation part of director changing rate by fluid vorticity, respectively. The influences between the molecules and the flow are described by 
the kinematic transport tensor $g$ and the viscous stress tensor $\sigma$ given as follows:
 The kinematic transport ${\bf g}$ is given by
\begin{align}\label{g}
{\bf g}=\gamma_1 N +\gamma_2D\n 
\end{align}
which represents the effect of the macroscopic flow field on the microscopic structure. The material coefficients $\gamma_1$ and $\gamma_2$ reflect the molecular shape and the slippery part between fluid and particles. The first term of ${\bf g}$ represents the rigid rotation of molecules, while the second term stands for the stretching of molecules by the flow.
The viscous (Leslie) stress tensor $\sigma$ has the following form
\begin{align}\label{sigma}\begin{split}
\sigma=& \alpha_1 (\n^TD\n)\n\otimes \n +\alpha_2N\otimes \n+ \alpha_3 \n\otimes N\\
&   + \alpha_4D + \alpha_5(D\n)\otimes \n+\alpha_6\n\otimes (D\n),
\end{split}
\end{align}
where $\mathbf{a}\otimes \mathbf{b}=\mathbf{a}\, \mathbf{b}^T$ for column vectors $\mathbf{a}$ and $\mathbf{b}$ in $\mathbb{R}^n$. These coefficients $\alpha_j$ $(1 \leq j \leq 6)$, depending on material and temperature, are called Leslie coefficients.
The following relations are assumed in the literature 
\begin{align}\label{a2g}
\gamma_1 =\alpha_3-\alpha_2,\quad \gamma_2 =\alpha_6 -\alpha_5,\quad \alpha_2+ \alpha_3 =\alpha_6-\alpha_5. 
\end{align} 
They also satisfy the following empirical relations 
\begin{align}\label{alphas}
&\alpha_4>0,\quad 2\alpha_1+3\alpha_4+2\alpha_5+2\alpha_6>0,\quad \gamma_1=\alpha_3-\alpha_2>0,\\
&  2\alpha_4+\alpha_5+\alpha_6>0,\quad 4\gamma_1(2\alpha_4+\alpha_5+\alpha_6)>(\alpha_2+\alpha_3+\gamma_2)^2\notag.
\end{align}

\subsection{Shear flows and informal statement of the results}

{Shear flows are of particular importance in modeling and analyzing fluids especially complex fluids where an extra structure exists. For instance, in liquid crystals the two unknowns are the flow $\bu$ and the molecules orientation $\n$. In such complex systems, simple flows, like shear or Poiseuille flows help understanding the interaction between the two structures in the material. This work extends two existing results:
\begin{itemize}
    \item In \cite{JDEDL}, the authors considered an oversimplified model and studied the existence of multiple stationary solutions and their linear stability. A similar bifurcation phenomenon was shown. This oversimplification allowed the authors to obtain more explicit formulas of solutions to the linear and nonlinear systems.
    \item in \cite{JDDEW}, the authors studied a more realistic model and successfully classified stationary solutions in terms of different ranges of certain physical parameters. Namely, $\gamma_1<|\gamma_2|, \gamma_1=|\gamma_2|$ and $\gamma_1>|\gamma_2|.$ The authors showed the existence of multiple solutions in a similar manner done in this work.
\end{itemize}
This work simultaneously generalizes \cite{JDEDL,JDDEW} in the sense that we consider the more physically realistic model and analyze the stationary solutions, multiplicity, bifurcation and stability. We only consider the case $\gamma_1=|\gamma_2|.$}
\\

For shear flows, with a choice of coordinates system, ${\bf u}$ and ${\bf n}$ take the form
\[{\bf u}(x,t)= (0,0, u(x,t))^T \;\mbox{ and }\;{\bf n}(x,t)=(\sin\theta(x,t),0,\cos\theta(x,t))^T,\]
where the motion ${\bf u}$ is along the $z$-axis and   the director ${\bf n}$ lies in the $xz$-plane with angle $\theta$ measured from the $z$-axis.   For this case of shear or Poiseulle flows, taking $\rho=1$ for simplicity, the EL  model is reduced to (see, e.g., \cite{CHL20}),
\begin{align}\label{sysf1}\begin{split}
     u_t&=\left(g(\theta)u_x+h(\theta)\theta_t\right)_x,\\
\nu\theta_{tt}+\gamma_1\theta_t&=c(\theta)\big(c(\theta)\theta_x\big)_x-h(\theta)u_x.
\end{split}
\end{align}
The functions $c(\theta)$, $g(\theta)$ and $h(\theta)$  are given by  
\begin{align}\label{fgh1}\begin{split}
 g(\theta):=&\alpha_1\sin^2\theta\cos^2\theta+\frac{\alpha_5-\alpha_2}{2}\sin^2\theta+\frac{\alpha_3+\alpha_6}{2}\cos^2\theta+\frac{\alpha_4}{2},\\
  h(\theta):=&\alpha_3\cos^2\theta-\alpha_2\sin^2\theta=\frac{\gamma_1+\gamma_2\cos(2\theta) }{2},\\
  c^2(\theta):=&K_1\cos^2\theta+K_3\sin^2\theta.
 \end{split}
 \end{align}
Very importantly, relations (\ref{a2g}) and (\ref{alphas}) imply that (see formula (2.4) in \cite{CHL20}),  for some constant $ \overline C>0$,
\begin{align}\label{positiveDamping}
g(\theta)\ge g(\theta)-\frac{h^2(\theta)}{\gamma_1}\ge \overline C>0.
\end{align}
{The property (\ref{positiveDamping}) plays crucial roles for the results in \cite{CHL20, CLS} for long time existence of solutions beyond singularity and  will be repeatedly applied in the proof of Theorem \ref{Stab} for linear stability of stationary solutions with small $\bar{u}$.}

{An extensive analysis has been devoted to the analysis of the EL models and a quite fruitful results have been obtained, including finite time blowups and long time existence beyond the blowup (see, e.g., \cite{chen2024initial,CHL20, chen2024singularity,CLS,CS22} and references therein).}

 In this work, we will consider shear flow of the  parabolic EL model with $\nu=0$. More precisely, we consider  an initial-boundary value problem of the system
\begin{align}\label{intro}\begin{split}
u_t&=\big(g(\theta)u_x+h(\theta)\theta_t\big)_x\\
    \gamma_1\theta_t&=c(\theta)(c(\theta)\theta_x)_x-h(\theta)u_x,
\end{split}
\end{align}
along with initial-boundary conditions
\begin{align}\label{ibv}\begin{split}
    &u(0,x)=u^{in}(x), \quad \theta(0,x)=\theta^{in}(x);\\
    &u(t,0)=0,\quad  u(t,1)=\bar u;\quad \theta(t,0)=\theta_0,\quad \theta(t,1)={\theta_1},
    \end{split}
\end{align}
for given functions $u^{in}(x)$ and $\theta^{in}(x)$ and constants $\theta_0,\;\theta_1\in\mathbb{R}$ and $\bar u>0.$


{
The parabolic Ericksen-Leslie model (with zero inertial coefficient $\nu=0$) as well as its approximation versions have been extensively
studied with a vast volume of results, particularly, on well-posedness of Cauchy problems including three dimension (see, e.g.\cite{lin2014recent} and reference therein). For an excellent survey on various topics regarding liquid crystals including model derivations and comparisons between models, we refer readers to the work \cite{lin2016global}. 
 Recently, in \cite{jiang2022zero},  the authors studied the zero inertia limit from the hyperbolic to parabolic Ericksen-Leslie
systems and, for small initial data, they established the convergence of solutions of the hyperbolic system
 to solutions of the parabolic system.}

In \cite{JDDEW}, a rather complete classification of the stationary solutions is provided. In this work, for the case where $\gamma_1=|\gamma_2|$ and $\theta_0=\theta_1$ in (\ref{ibv})
\begin{itemize}
\item[(i)] we recapture the analysis on multiplicity of stationary solutions from \cite{JDDEW}: in particular, it is shown that,
 for small $\bar u$, there is a unique stationary solution,  for large $\bar u>0$, there are  multiple  stationary solutions through a saddle-node bifurcation;

\item[(ii)] the stability of the stationary solutions near the bifurcation moment is examined and a quantitative condition is discovered that determines which of the bifurcating solution is stable and which is unstable;

\item[(iii)] as a relevant portion of this study, we prove, for small $\bar u$, the unique stationary solution   is linearly stable.
\end{itemize}

 We also make remarks for the case $\gamma_1>|\gamma_2|$ and explain the similarity as well as some differences to the case studied here. The case $\gamma_1<|\gamma_2|$ is more involved mostly due to the appearance of different types of stationary solutions {associated with periodic and heteroclinic orbits of the phase space portrait of the stationary systems.} It would be very interesting to explore the expected rich dynamical phenomena in the latter case.

\medskip

The rest of this paper is organized as follows. In Section \ref{onSS} we analyze the steady state system and re-derive the existence of multiple stationary solutions for large enough $\bar u$. In Section \ref{BifCri}, guided by the findings in Section \ref{onSS} and via careful and sharp computations, {we investigate the bifurcation moment of the stationary solutions and establish the bifurcation of the zero eigenvalue associated with the linear system. This is accomplished by using Evans function, but more importantly, by exploiting the relation of integrability of the linearized system with that of nonlinear system for stationary solutions.} In Section \ref{SpecStab}, we prove the sufficient condition to distinguish the linearly stable/unstable stationary solutions bifurcating from a stationary solution with the zero eigenvalue. Finally, in Section \ref{SmallStab}, we prove the linear stability of stationary solutions with small sheer speed.


\section{Multiplicity of steady states for $\gamma_1=|\gamma_2|$ with $\theta_0=\theta_1$}\label{onSS}
\setcounter{equation}{0}

A steady state $(u(x),\theta(x))$ of (\ref{intro}) and (\ref{ibv}) with $\theta_0=\theta_1$  satisfies  
\begin{align}\label{ssth}
\big(g(\theta)u_x\big)_x=0,\quad c(\theta)(c(\theta)\theta_x)_x-h(\theta)u_x=0
\end{align}
and
\begin{align}\label{bc}
    u(0)=0,\; u(1)=\bar u>0;\quad
        \theta(0)=\theta(1)=\theta_0.
\end{align}

From the first equation \eqref{ssth}, one has
\begin{align}\label{varM}
 g(\theta)u_x=M^2,\;\mbox{ or equivalently, }\;   u(x)=M^2\int_0^x\frac{1}{g(\theta(s))}\,ds,
\end{align}
for some $M>0$ to be determined. Substituting \eqref{varM} into \eqref{ssth}, one obtains
\begin{align}\label{2nd}
    c(\theta)(c(\theta)\theta_x)_x-\frac{h(\theta)}{g(\theta)}M^2=0.
\end{align}

Equation  \eqref{2nd} can be rewritten as a system of first order ODEs
\begin{align}\begin{split}\label{storder}
   \theta_x&=\frac{\eta}{c^2(\theta)},\quad
   \eta_x=\frac{c'(\theta)}{c^3(\theta)}\eta^2+\frac{h(\theta)}{g(\theta)}M^2,
   \end{split}
\end{align}
which is a Hamiltonian system with a Hamiltonian function
\begin{align}\label{HamFun}
    H(\theta, \eta)=\frac{1}{2c^2(\theta)}\eta^2-{M^2}G(\theta)\;\mbox{ where }\; G(\theta)=\int_{\theta_0}^\theta\frac{h(s)}{g(s)}ds.
\end{align}

One can then get the phase portrait easily (see Figure~\ref{fig:phase}).
{Recall that the phase portrait is a collection of ``representative" orbits of the dynamical system (2.5); that is, a collection of {\em parametric} curves $(\theta(x), \eta(x))$ of the solution of (2.5). It should be pointed out that, although the independent variable $x$ is not represented in the orbits, near equilibria (cusp-type when $\gamma_1 = |\gamma_2|$), the orbits take more $x$-range to pass around equilibria. Since, for the BVP, $x\in [0,1]$, thus $M$ or $\bar{u}$ must be large if the orbit for the BVP is close to the cusps. It turns out the presence of the cusps is exactly the underlying reason for the property (ii) in Proposition 2.2, which is the key for the main results, Theorem 2.3 and Corollary 2.4, for the stationary solutions in this section.  }
Note that, for $\gamma_1=|\gamma_2|$, the system \eqref{storder} has a set of equilibria $\{(e_n,0): n\in\Z\}$ with $e_n=\frac{2n+1}{2}\pi$ for $\gamma_1=\gamma_2$ and $e_n=n\pi$  for $\gamma_1=-\gamma_2$. In both cases, the equilibria are {\em cusps}. 



\begin{figure}
\centering
\includegraphics[width=0.8 \textwidth]{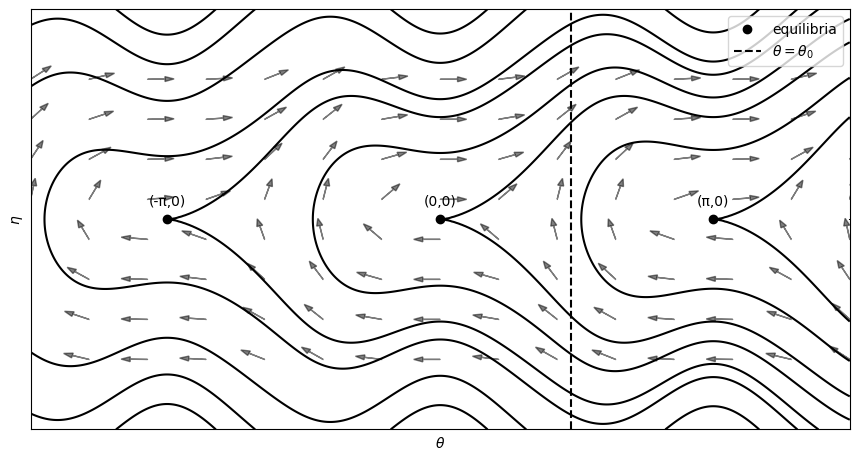}
\caption{\em The phase portrait for the case $\gamma_1=-\gamma_2,\, (\gamma_2<0)$ For any $n\in \mathbb{N}\cup \{0\},$ the point $(\pm n\pi,0)$ is an equilibrium point (cf. equations \eqref{storder}). A solution to the system is a trajectory starting at $x=0$ from the vertical dashed line (at $\theta=\theta_0)$ following the vector field and hitting the same vertical dashed line ($\theta=\theta_0)$ at $x=1.$ } 
\label{fig:phase}
\end{figure}

Suppose $\theta_0\not\in\{e_n\}_{n\in\Z}$ (Otherwise, some minor issues may disrupt the discussion a bit). Let $(\theta(x),\eta(x))$ be a solution of the stationary boundary value problem (BVP) \eqref{storder} and \eqref{bc} and let $(\tilde\theta, 0)$ be the intersection of the orbit with the $\theta$-axis. 
Due to the symmetry of the BVP about the $\theta$-axis, one has $(\tilde\theta, 0)=(\theta(1/2),\eta(1/2))$.
 
Applying the Hamiltonian first integral, one has 
\begin{align*}
  \frac{\eta^2}{2c^2(\theta)}-{M^2}G(\theta)=-{M^2}G(\tilde\theta)
\end{align*}
  For $x\in (1/2,1)$,
\begin{align}\label{eq4theta}
    \frac{d\theta}{dx}=\frac{\sqrt{2}M}{c(\theta)}\sqrt{G(\theta)-G(\tilde\theta)}.
\end{align}
Separating  variables in  \eqref{eq4theta} and integrating from $x=1/2$ to $x=1$, one has
\begin{align}\label{box1}
M=\sqrt{2}\int_{\tilde\theta}^{{\theta_0}}\frac{c(\theta)}{\sqrt{G(\theta)-G(\tilde\theta)}}\,d\theta.
\end{align}
On the other hand, it follows from \eqref{varM} that
\begin{align}
   \displaystyle M^2=\frac{\bar u}{\int_0^1\frac{1}{g(\theta(s))}ds}, 
\end{align}
which can be rewritten as
\begin{align}\label{Mexp}
   \displaystyle M^2=\frac{\bar u}{\sqrt{2}\int_{\tilde\theta}^{{\theta_0}}\frac{c(\theta)}{M\,g(\theta)\sqrt{G(\theta)-G(\tilde\theta)}}\,d\theta},
\end{align}
where the following change of variable is used
\begin{align*}
 \int_0^1\frac{1}{g(\theta(s))}ds=2\int_{1/2}^{1}\frac{1}{g(\theta(s))}ds=\sqrt{2}\int_{\tilde\theta}^{{\theta_0}}\frac{c(\theta)}{M\,g(\theta)\sqrt{G(\theta)-G(\tilde\theta)}}\,d\theta.
\end{align*}
Combining \eqref{box1} and \eqref{Mexp}, 
we obtain
\begin{align}\label{Eq4SS}\begin{split}
    \bar u&=\sqrt{2}\,M\int_{\tilde\theta}^{{\theta_0}}\frac{c(\theta)}{g(\theta)\sqrt{G(\theta)-G(\tilde\theta)}}\,d\theta\\
    &=2\int_{\tilde\theta}^{{\theta_0}}\frac{c(\theta)}{\sqrt{G(\theta)-G(\tilde\theta)}}\,d\theta\,\int_{\tilde\theta}^{{\theta_0}}\frac{c(\theta)}{g(\theta)\sqrt{G(\theta)-G(\tilde\theta)}}\,d\theta.
    \end{split}
\end{align} 
\begin{lem}\label{invG}
For $\gamma_1= |\gamma_2|,$ the function $G$ is invertible.
\end{lem}
\begin{proof} It follows from the definition of $G$ in \eqref{HamFun} 
\[G'(\theta)=\frac{h(\theta)}{g(\theta)}.\]
Then $g(\theta)>0$ from \eqref{positiveDamping} and $h(\theta)=(\gamma_1+\gamma_2\cos(2\theta))/2$ yield that, for $\gamma_1=|\gamma_2|$,
$G'(\theta)\ge 0$, and the equal sign occurs only at $e_n=n\pi$ or $e_n=\frac{2n+1}{2}\pi$ for $n\in\Z$ depending on $\gamma_1=-\gamma_2$ or $\gamma_1=\gamma_2$ respectively.
The claim then follows.
\end{proof}
Denote the inverse of $G$ by $F.$ Set  $\beta=-G(\tilde\theta)$. Then $\beta>0$. Make a  change of variables $s=-G(\theta)$ (or equivalently, $\theta=F(-s)$) for the integrals in  \eqref{Eq4SS} to get 
\begin{align}\label{ubar}\begin{split}
    \frac{\bar u}{2}&=\int_0^\beta\frac{c(F(-s))F'(-s)}{\sqrt{\beta-s}}\,ds\int_0^\beta\frac{c(F(-s))F'(-s)}{g(F(-s))\sqrt{\beta-s}}\,ds\\
    &=\int_0^\beta\frac{c(F(-s))g(F(-s))}{h(F(-s))\sqrt{\beta-s}}\,ds\int_0^\beta\frac{c(F(-s))}{h(F(-s))\sqrt{\beta-s}}\,ds\\
    &=\int_0^\beta\frac{cg}{h}(F(t-\beta))\frac{1}{\sqrt{t}}\,dt\int_0^\beta\frac{c}{h}(F(t-\beta))\frac{1}{\sqrt{t}}\,dt\\
    &=:D(\beta).
    \end{split}
\end{align}
Recall that  $\{e_n\}_{n\in\Z}$ is the set of points where $h(e_n)=0.$ Set $\beta_n=-G(e_n)$ .
\begin{prop}\label{lemmaD} For $\gamma_1=|\gamma_2|$,  the function $D(\beta)$ is defined on $\R_+\backslash\{\beta_n\}_{n\in\Z}$. Furthermore,   one has (see Fig. \ref{fig:D})
\begin{itemize}
\item[(i)] $D(\beta)>0$  on its domain, $D(0^+)=0$, $D(\infty)=\infty$, and 
\item[(ii)]    $D(\beta)\to +\infty$ as $\beta\to \beta_n$ for $n\in \Z$.
\end{itemize}  
     \end{prop}
     \begin{proof} (i). Since $\gamma_1=|\gamma_2|,$ we have $h\geq 0.$ Hence for $\beta>0,$ we have $D(\beta)>0.$ Clearly, $D(0^+)=0$ and $D(\infty)=\infty.$
     
(ii). Recall that $e_n$ is a point where $h(e_n)=0.$ By the change of variables above, $s=-G(\theta)$ and then $\beta-s=t,$ it suffices to show that,
 as $a\to e_n$, either
\begin{align}
        \int_{a}^{\theta_0}\frac{c(z)}{\sqrt{G(z)-G(a)}}\,dz\to \infty\quad\text{or}\quad \int_{a}^{\theta_0}\frac{c(z)}{g(z)\sqrt{G(z)-G(a)}}\,dz\to \infty.
    \end{align}

In fact, both integrals go to infinity. We will only prove it for the first integral and the proof for the second is exactly the same. \\

 We expand ${G(z)}$ around $e_n$ and use $G'(e_n)=h(e_n)/{g(e_n)}=0$ to have
\begin{align*}
    G(z)&=G(e_n)+G'(e_n)(z-e_n)+\frac{G''(e_n)}{2!}(z-e_n)^2+...\\
 &=G(e_n)+\frac{G''(e_n)}{2!}(z-e_n)^2+...\\
 &=G(e_n)+\sum_{p=2}^{\infty}\frac{G^{(p)}(e_n)}{p!}(z-e_n)^p
\end{align*}
Now we rewrite the integral,
\begin{align}
\displaystyle\bigintsss_{e_n}^{\theta_0}\displaystyle\frac{c(z)}{\sqrt{G(z)-G(e_n)}}\,dz=\displaystyle\bigintsss_{e_n}^{\theta_0}\displaystyle\frac{c(z)}{\displaystyle\sqrt{\sum_{p=2}^{\infty}\frac{G^{(p)}(\tilde{\theta})}{p!}(z-e_n)^p}}\,dz.
\end{align}
Since $G''(e_n):=(\frac{h}{g})'(e_n)\neq 0,$ then we can write
\begin{align*}
\bigintsss_{e_n}^{\theta_0}\displaystyle\frac{c(z)}{\sqrt{\sum_{n=2}^{\infty}\frac{G^{(p)}(e_n)}{p!}(z-e_n)^p}}\,dz=\bigintsss_{e_n}^{\theta_0}\displaystyle\frac{c(z)}{(z-e_n)\displaystyle\sqrt{\sum_{p=2}^{\infty}\frac{G^{(p)}(e_n)}{p!}(z-e_n)^{p-2}}}\,dz.
\end{align*}
It is clear that the integral diverges,  that is, $D(\beta)\to+\infty,$ as $\beta\to \beta_n.$
\end{proof}

\begin{figure}[h]
\centering
\includegraphics[width=0.4\textwidth]{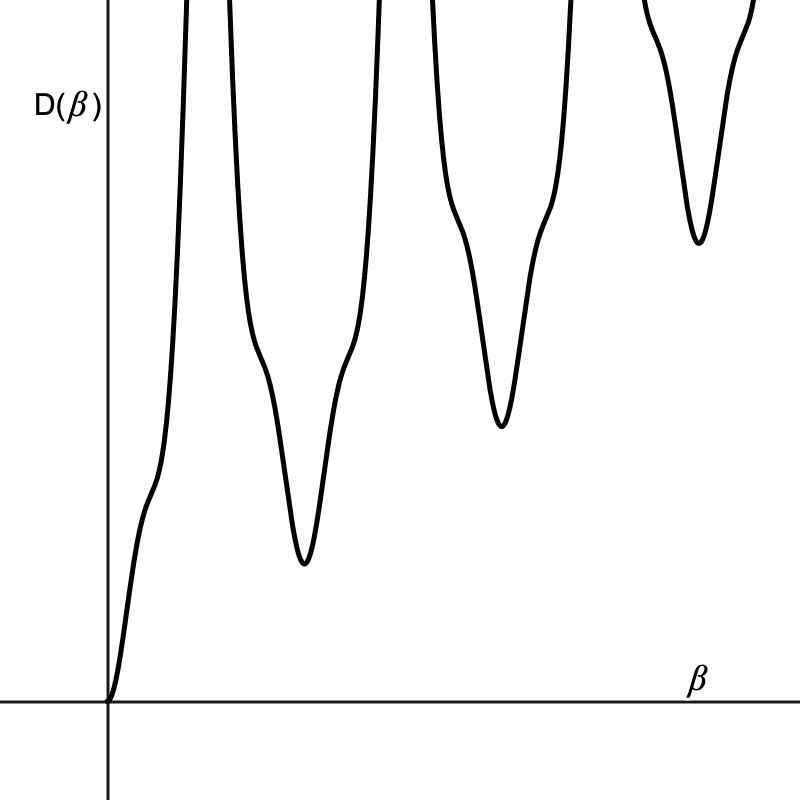}
\caption{\em A sketch of the graph of $D(\beta)$ for $\gamma_1=|\gamma_2|$ (See Prop. \ref{lemmaD}).}
\label{fig:D}
\end{figure}
Therefore, we have the following characteristic for stationary solutions.
\begin{thm}\label{propD} For $\gamma_1= |\gamma_2|$,
there is a one-to-one correspondence between stationary solutions $(u,\theta)$ of \eqref{ssth}-\eqref{bc} and solutions $\beta$ of $2D(\beta)=\bar u$. 
\end{thm}
\begin{proof}
    The proof follows from the above construction of the map $D(\beta)$.
\end{proof}


     Combining Theorem \ref{propD} and Proposition \ref{lemmaD}, we conclude that 
    \begin{cor}\label{multi} For $\gamma_1=|\gamma_2|$, if $\bar u$ is large enough, then the stationary problem \eqref{ssth}-\eqref{bc} has multiple solutions.
    \end{cor}
 
 { 
\begin{rem}\label{SNbif4D} The basic mechanism for the existence of multiple stationary solutions is the well-known {\em saddle-node bifurcation} for the nonlinear equation $D(\beta)=\bar{u}/2$ with $\bar{u}>0$ as a parameter. More precisely, for each interval $I_n:=(\beta_{n},\beta_{n+1})$ with $n\ge 1$,  there is a critical value 
\begin{align}\label{cribaru}
\bar{u}_n=\min\{D(\beta): \beta\in I_n\},
\end{align}
  so that, (i) for $\bar{u}<\bar{u}_n$, there is No solution for $D(\beta)=\bar{u}/2$ over $I_n$, (ii) at the critical moment $\bar{u}=\bar{u}_n$, there is generically a Double root  for $D(\beta)=\bar{u}/2$, and (iii) for $\bar{u}>\bar{u}_n$ (but close), there are (at least) two solutions for $D(\beta)=\bar{u}/2$ in $I_n$.
\end{rem}

This mechanism for creation of multiple stationary solutions strongly suggests a corresponding {\em saddle-node bifurcation} for time evolution infinite-dimensional dynamical system \eqref{intro} and \eqref{ibv}, that is,  in addition to the above described saddle-node bifurcation of stationary solutions, the stability also exhibits a saddle-node bifurcation. More precisely, {for each $n\ge 1$,}  (i) at the bifurcation moment $\bar{u}=\bar{u}_n$, the corresponding stationary solution should have a zero eigenvalue for the linearization of the initial-boundary value problem \eqref{intro} and \eqref{ibv} and, (ii) after the bifurcation, {that is, for $\bar{u}>\bar{u}_n$ and $|\bar{u}-\bar{u}_n|\ll 1$,} the zero eigenvalue becomes negative for one of the two bifurcated stationary solutions and becomes positive for the other bifurcated stationary solution. The main result of this work  establishes this statement and will be presented in the remaining sections.


}

\begin{rem}\label{genCase}
 \begin{itemize}
     \item [(i)] The case $\gamma_1=|\gamma_2|$ {has two subcases:}  $\gamma_1=-\gamma_2$ and $\gamma_1=\gamma_2.$ However, the analysis and results for both subcases are essentially identical. The only minor difference is that the set of equilibria $\{(e_n,0)\}_{n\in\mathbb{N}}$ (or the set of zeros of $h(\theta)$) is shifted by $\pi/2$ which accordingly shifts the graph of $D(\beta).$ That is, the set of points $\{\beta_n\}_{n\in\mathbb{N}},$ where $D(\beta)$ blows up, is redefined to be $\{\beta_n\}_{n\in\mathbb{N}}:=\{-G(e_n+\frac{\pi}{2})\}_{n\in\mathbb{N}}.$ All the results apply.
     \item[(ii)]  For $\gamma_1>|\gamma_2|$, we compare relevant quantities with those for $\gamma_1=|\gamma_2|$. First of all, the equilibria of the system (\ref{storder}) disappear, since $h(\theta)>0$ for all $\theta\in\mathbb{R}$. Lemma \ref{invG} still holds with strict inequality $G'(\theta)>0.$ Consequently, for $\gamma_1>|\gamma_2|$ but close,  the vertical asymptotes of $D(\beta)$ become local maxima. For this situation, saddle-node bifurcations occur also at local maxima.  Of course, it is possible that $D(\beta)$ becomes monotone over a certain range of parameter space. 
 \end{itemize}
\end{rem}

\section{Characteristic of bifurcation moment of stead states}\label{BifCri}
\setcounter{equation}{0}
We will establish that, at the bifurcation moment $\bar{u}=\bar{u}_n$ for $n\ge 1$ defined in \eqref{cribaru}, $\lambda=0$ is an eigenvalue of the linearization of the problem \eqref{intro} and \eqref{ibv} along the corresponding stationary solution, as
  remarked at the end of the previous section (See Figure \ref{fig:D&u}). The claim will be proved by analyzing Evans function for the eigenvalue problem.
\begin{figure}[h]
\centering
\includegraphics[width=0.6\textwidth]{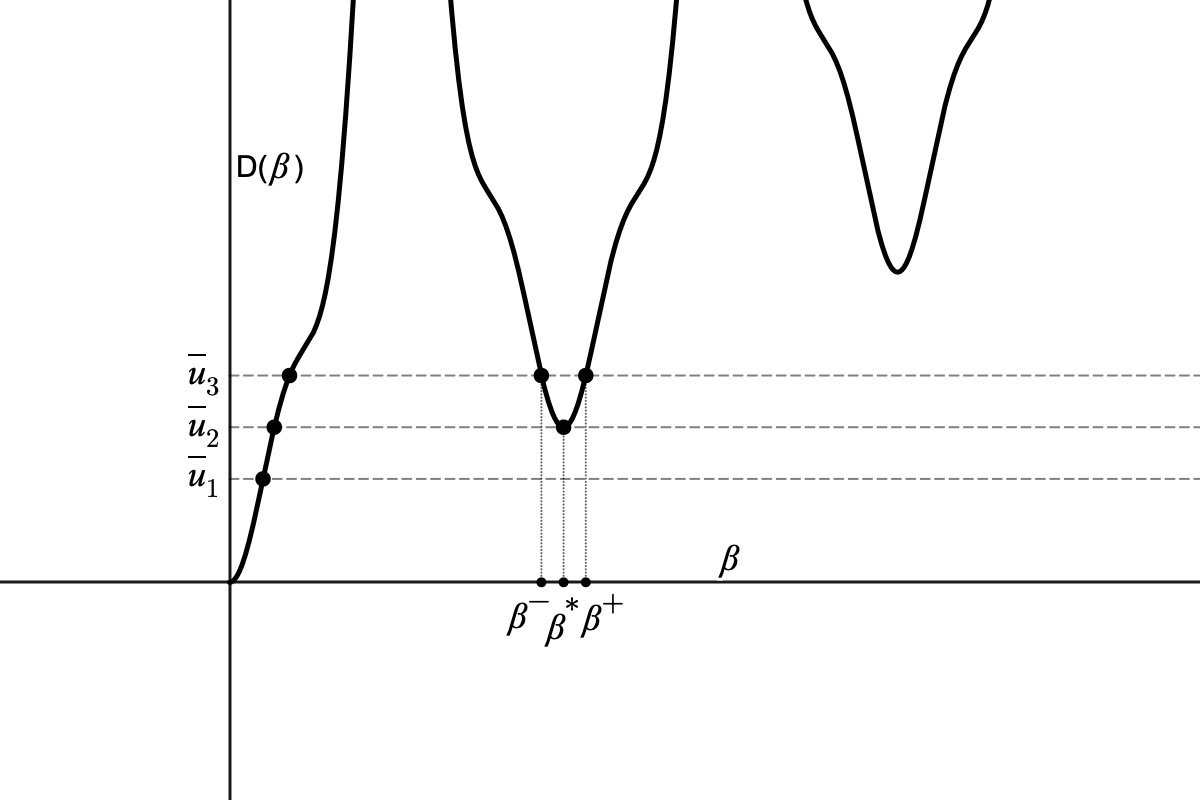}
\caption{\em 
A saddle-node bifurcation  for the equation $D(\beta)=\bar{u}/2$ takes place as $\bar u$ crosses a local minimum of $D(\beta)$, say at $\beta^*.$}
\label{fig:D&u}
\end{figure}

\subsection{Linearization and eigenvalue problems}
 Let $(u,\theta)=(u(x),\theta(x))$ be a stationary solution of \eqref{ssth} and \eqref{bc}.  We linearize \eqref{intro} around  $(u(x),\theta(x))$   by looking for first order time dependent perturbations:
\begin{align*}
    u(t,x)&=u(x)+\varepsilon \tilde{U}(t,x)+O(\varepsilon^2),\quad
    \theta(t,x)=\theta(x)+\varepsilon \tilde{\Theta}(t,x)+O(\varepsilon^2),
\end{align*}
where $\tilde{U}(t,0)=\tilde{U}(t,1)=\tilde{\Theta}(t,0)=\tilde{\Theta}(t,1)=0.$ 
Substituting $u$ and $\theta$  into \eqref{intro} and keeping the terms of order $\varepsilon$, one has
\begin{align}
\tilde{U}_t=&\big[g(\theta)\tilde{U}_x\big]_x+\big[g'(\theta)\tilde{\Theta} u_x\big]_x+\big[h(\theta)\tilde{\Theta}_t\big]_x,\label{Upert}\\
 \gamma_1\tilde{\Theta}_t=&c'(\theta)(c(\theta)\theta_x)_x\tilde{\Theta}+c(\theta)(c(\theta)\tilde{\Theta})_{xx}-h'(\theta)\tilde{\Theta} u_x-h(\theta)\tilde{U}_x. \label{phipert}
\end{align}
For eigenvalue problems, set $\tilde{U}(t,x)\equiv e^{\lambda t}U(x)$ and $\tilde{\Theta}(t,x)\equiv e^{\lambda t}\Theta(x)$ to get
\begin{align}\label{OriEvp}\begin{split}
\lambda U=&\big[g(\theta){U}_x\big]_x+\big[g'(\theta){\Theta} u_x\big]_x+\big[\lambda h(\theta){\Theta}\big]_x,\\
\lambda \gamma_1{\Theta}=&c'(\theta)(c(\theta)\theta_x)_x{\Theta}+c(\theta)(c(\theta){\Theta})_{xx}-h'(\theta){\Theta} u_x-h(\theta){U}_x.
 \end{split}
\end{align}
 Introduce
\begin{align*}
    P(x)&:=g(\theta)U_x+{g'(\theta)}u_x\Theta+\lambda{h(\theta)}\Theta\;\mbox{ and }\; Q(x):=\big(c^2(\theta)\Theta\big)_x.
\end{align*}
Then the eigenvalue problem \eqref{OriEvp} can be  written as a first order ODE system 
\begin{align}\label{evp}
Z'=\big(A(x)+\lambda B(x)\big)Z,
\end{align}
where $Z=(U,P,\Theta,Q)^T$, 
\begin{align*}
 &A(x)=
 \begin{pmatrix}
0 & \frac{1}{g(\theta)} & -\frac{g'(\theta)u_x}{g(\theta)} & 0\\
0 & 0 & 0 & 0\\
0 & 0 & -2\frac{c'(\theta)\theta_x}{c(\theta)} & \frac{1}{c^2(\theta)}\\
0 & \frac{h(\theta)}{g(\theta)} & b_1(\theta,\theta_x,u_x) & 2\frac{c'(\theta)\theta_x}{c(\theta)}
\end{pmatrix},\\
&\\
&B(x)=
 \begin{pmatrix}
0 & 0 & -\frac{h(\theta)}{g(\theta)} & 0\\
1 & 0 & 0 & 0\\
0 & 0 & 0 & 0\\
0 & 0 & b_2(\theta) & 0
\end{pmatrix},\\
&\\
& b_1(\theta,\theta_x,u_x)=\Big(\frac{h}{g}\Big)'(\theta)g(\theta)u_x+\Big({c(\theta)}c''(\theta)-3c'^2(\theta)\Big){\theta_x}^2,\\
 &b_2(\theta)=\gamma_1-\frac{h^2(\theta)}{g(\theta)}. 
 \end{align*}
We recall that $\lambda$ is an eigenvalue associated to $(u(x),\theta(x))$ if \eqref{evp} has a non-zero solution $Z(x)=(U(x),P(x),\Theta(x),Q(x))^T$ with $U(0)=U(1)=\Theta(0)=\Theta(1)=0.$\\

\subsection{Evans function and an explicit criterion for zero eigenvalue}
Let $\beta$ correspond to the stationary solution $(u(x),\theta(x))$ from Theorem \ref{propD}.   For $j=1,2,3,4$, let $Z_j(x)=Z_j(x;\lambda,\beta)$  be solutions of \eqref{evp} with
\begin{align*}
    Z_1(0)=(0,1,0,0)^T,\quad Z_2(0)=(0,0,0,1)^T;\\
    Z_3(1)=(0,1,0,0)^T,\quad Z_4(1)=(0,0,0,1)^T.
\end{align*}
Note that $Z_1(x)$ and $Z_2(x)$ satisfy the boundary condition $U(0)=\Theta(0)=0$, and $Z_3(x)$ and $Z_4(x)$ satisfy $U(1)=\Theta(1)=0$.\\

Define the Evans function as
\begin{align}
    E(x,\lambda,\beta):=\det
    (Z_1(x)\,|\,Z_2(x)\,|\,Z_3(x)\,|\,Z_4(x)).
\end{align}
Since $\text{tr}(A(x)+\lambda B(x))=0,$ by Liouville's formula, $E(x,\lambda,\beta)$ is independent of $x$.
In the following,  we write $E(\lambda,\beta)$ for $E(x,\lambda,\beta)$. The next statement is well-known (See, e.g., \cite{JDEDL}).
\begin{prop}\label{ev}
  $\lambda$ is an eigenvalue   if and only if $E(\lambda,\beta)=0$. 
\end{prop}
Note that at a saddle-node bifurcation moment of stationary solutions, $\lambda=0$ is expected to be an eigenvalue. 
The goal is to evaluate Evans function at $\lambda=0$, which amounts to {a sufficiently good understanding of the linear system (\ref{evp}).} The observation is that the nonlinear problem for stationary solutions can be solved very explicitly and the derivative of the flow map/solution operator with respect to initial conditions solves the linearized system. As a result, we are able to relate Evans function at $\lambda=0$  directly to $D'(\beta)$. In fact,  
{
\begin{thm}\label{e=d'} One has 
\begin{align}\label{ed'}
    E(0,\beta)=-\frac{2\beta}{p_0c^2(\theta_0)}D'(\beta).
\end{align}
Consequently, 
   $\lambda=0$ is an eigenvalue if and only if $D'(\beta)=0.$ 
\end{thm}
 The formula (\ref{ed'}) is the main analytical contribution of this work and will be established through several steps in the remaining part of this subsection.

{First of all, we will briefly explain our idea  to relate $E(0;\beta)$ with $D'(\beta)$. Note that, for $\lambda=0$, the system  (3.4) is nothing but the {\em linearized} system of the stationary system along the stationary solution. On the other hand, it is known that the fundamental matrix solution $\Phi(x)$ of the linearized system is the Jacobian matrix $\frac{D}{Dz}\phi^x(z)$  of the solution operator $\phi^x$ of the stationary solution with respect to initial conditions $z$ (see formula (3.14)). Lastly, Theorem 2.3 and its proof relate the solution of BVP of $\phi^x$  directly with the mapping $D(\beta)$. These are the main reasons for a relation like in Theorem 3.2 between $E(0,\lambda)$ and $D'(\beta)$. Of course, it is still quite involved in carrying out this idea. We will split the proof of Theorem 3.2 into four steps in the rest of this subsection.}
 
{\textbf{Step 1:}} We will begin with some general preparations.}
For $\lambda=0$,  system \eqref{evp} becomes
\begin{align}\begin{split}\label{sys0}
    U_x&=\frac{1}{g(\theta)}P-\frac{g'(\theta)}{g(\theta)}u_x\Theta;\\
    P_x&=0;\\
    \Theta_x&=-2\frac{c'(\theta)\theta_x}{c(\theta_x)}\Theta+\frac{1}{c^2(\theta)}Q;\\
    Q_x&=\frac{h(\theta)}{g(\theta)}P+[c''(\theta)c(\theta){\theta_x}^2-3c'^2(\theta){\theta_x}^2+(\frac{h}{g})'(\theta)p]\Theta+2\frac{c'(\theta)}{c(\theta)}\theta_xQ.
    \end{split}
\end{align}
This is the linearization of the first order nonlinear stationary system recast below
\begin{align}
\begin{split}\label{nlss}
    u_x=&\frac{p}{g},\quad 
    p_x=0;\\
    \theta_x=&\frac{1}{c^2(\theta)}\eta,\quad
    \eta_x=\frac{c'(\theta)}{c^3(\theta)}{\eta}^2+\frac{h}{g}(\theta)p,
    \end{split}
\end{align}
where \begin{align*}
    p(x)&:=g(\theta)u_x;\quad \eta(x):=c^2(\theta)\theta_x.
\end{align*}

{\textbf{Step 2:}} The next result can be verified directly.
 \begin{prop}\label{int4nonl} Let {$z(x)=(u(x),p(x),\theta(x),\eta(x))$} be the solution of system \eqref{nlss} with the initial condition $z(0)=(u_0,p_0,\theta_0,\eta_0)$. Then one has the following conservative quantities:
\begin{align*}
        H_1(p)=&p=p_0;\\
        H_2(p,\theta,\eta)=&\frac{1}{2c^2}\eta^2-pG(\theta)=\frac{1}{2c^2(\theta_0)}\eta_0^2=-p_0G(\tilde\theta); \; (p_0>0\; \mbox{ since }\; G(\tilde\theta)<0)\\
        H_3(u,p,\theta)=&u-\sqrt{\frac{p}{2}}\bigg[\chi_{[x\leq 1/2]}S(\theta,\theta_0;\tilde\theta)+\chi_{[x>1/2]}\big(S(\tilde\theta,\theta_0;\tilde\theta)+S(\tilde\theta,\theta;\tilde\theta)\big)\bigg]=u_0
\end{align*}
where
\begin{align}\nonumber
S(X,Y;\tilde\theta)&=\int_{X}^{Y}\frac{c(z)}{g(z)\sqrt{G(z)-G(\tilde\theta)}}\,dz.
\end{align}
\end{prop}
Note that the conservation of $H_2$ gives an expression of the dependence between $\tilde\theta$ and $p_0,\eta_0,\theta_0,$ namely,
\begin{align}\label{tilderelation}
    \frac{1}{2c^2(\theta_0)}\eta_0^2=-p_0G(\tilde\theta(p_0,\theta_0,\eta_0)).
\end{align}

{\textbf{Step 3:}} Given a condition $z(0)=z_0:=(u_0,p_0,\theta_0,\eta_0)$ at $x=0$, we use the quantities $H_1,H_2,$ and $H_3$  to solve the non-linear system implicitly 
\begin{align}\begin{split}\label{nls}
    u(x)=\displaystyle u_0+\sqrt{\frac{p_0}{2}}\bigg[\chi_{[x\leq1/2]}S(\theta,\theta_0;\tilde\theta)&+\chi_{[x>1/2]}\big(S(\tilde\theta,\theta_0;\tilde\theta)+S(\tilde\theta,\theta;\tilde\theta)\big)\bigg],\\
    p&=p_0,\\
\chi_{[x>1/2]}\big(S^g(\tilde\theta,\theta_0;\tilde\theta)+S^g(\tilde\theta,\theta;\tilde\theta)\big)&+\chi_{[x\leq1/2]}S^g(\theta,\theta_0;\tilde\theta)
=x\sqrt{2p_0},\\
\eta(x)=\displaystyle c(\theta(x))&\sqrt{\frac{\eta^2_0}{c^2(\theta_0)}+2p_0G(\theta(x))},
\end{split}
\end{align}
where
\begin{align}\label{SSg}\begin{split}
S(X,Y;\tilde\theta)=&\int_{X}^{Y}\frac{c(z)}{g(z)\sqrt{G(z)-G(\tilde\theta)}}\,dz,\\
S^g(X,Y;\tilde\theta)=&\int_{X}^{Y}\frac{c(z)}{\sqrt{G(z)-G(\tilde\theta)}}\,dz.
\end{split}
\end{align}
{ 
It follows from  \eqref{tilderelation} that
    \begin{align}\label{dpeta}\begin{split}
        \frac{h(\tilde{\theta})}{g(\tilde{\theta})}\frac{\partial\tilde{\theta}}{\partial p_0}=&\frac{-G(\tilde{\theta}(p_0,\theta_0,\eta_0)}{p_0}:=\frac{\beta}{p_0},\\
        \frac{h(\tilde{\theta})}{g(\tilde{\theta})}\frac{\partial\tilde{\theta}}{\partial \eta_0}=&-\frac{\eta_0}{p_0c^2(\theta_0)}=\frac{\sqrt{-G(\tilde{\theta})}}{c(\theta_0)}\sqrt{\frac{2}{p_0}}:=\frac{\sqrt{\beta}}{c(\theta_0)}\sqrt{\frac{2}{p_0}}.
        \end{split}
        \end{align}
 Similarly, from \eqref{SSg}, we have
        \begin{align}\label{dtheta}\begin{split}
        \frac{\partial S(\tilde\theta,\theta_0;\tilde\theta)}{\partial \tilde\theta}=&\frac{h(\tilde{\theta})}{g(\tilde{\theta})}\int_{\tilde{\theta}}^{\theta_0} (\frac{c}{h})'(z)\frac{1}{\sqrt{G(z)-G(\tilde{\theta})}}\,dz\\
&        -\frac{c(\theta_0)}{h(\theta_0)}\frac{h(\tilde{\theta})}{g(\tilde{\theta})}\frac{1}{\sqrt{G(\theta_0)-G(\tilde{\theta})}},\\
        \frac{\partial S^g(\tilde\theta,\theta_0;\tilde\theta)}{\partial \tilde\theta}=&\frac{h(\tilde{\theta})}{g(\tilde{\theta})}\int_{\tilde{\theta}}^{\theta_0} (\frac{cg}{h})'(z)\frac{1}{\sqrt{G(z)-G(\tilde{\theta})}}\,dz\\
       & -\frac{c(\theta_0)g(\theta_0)}{h(\theta_0)}\frac{h(\tilde{\theta})}{g(\tilde{\theta})}\frac{1}{\sqrt{G(\theta_0)-G(\tilde{\theta})}}.
        \end{split}
    \end{align}
}
Denote the solution operator of the stationary system \eqref{nlss} by $\phi^x.$ We have 
\[\phi^x(z_0)=z(x;z_0):=(u(x),p(x),\theta(x),\eta(x)).\] 
It follows from \eqref{tilderelation}  and \eqref{nls}, etc., that the solution operator (matrix) of the corresponding linear system (\ref{sys0}) is
\begin{align}
  \Phi(x):=\frac{D}{Dz_0}\phi^x(z_0)= \begin{pmatrix}
      1 & T_{12} & T_{13} & T_{14}\\
      0 & 1 & 0 & 0\\
      0 & T_{32} & T_{33} & T_{34}\\
      0 & T_{42} & T_{43} & T_{44}
  \end{pmatrix}
\end{align}
where
\begin{align}\label{TT}\begin{split}
    T_{12}:=\frac{\partial u}{\partial p_0} \quad T_{13}:=\frac{\partial u}{\partial \theta_0} \quad T_{14}:=\frac{\partial u}{\partial \eta_0},\\
    T_{32}:=\frac{\partial \theta}{\partial p_0} \quad T_{33}:=\frac{\partial \theta}{\partial \theta_0} \quad T_{34}:=\frac{\partial \theta}{\partial \eta_0},\\
    T_{42}:=\frac{\partial \eta}{\partial p_0} \quad T_{43}:=\frac{\partial \eta}{\partial \theta_0} \quad T_{44}:=\frac{\partial \eta}{\partial \eta_0}.
    \end{split}
\end{align}
Hence, the solution $Z(x):=(U(x),P(x),\Theta(x),Q(x))^T$ to the linearized system can be written as
\begin{align}
    Z(x)=\Phi(x)\,Z_0=\begin{pmatrix}
        U_0+P_0T_{12}+\Theta_0T_{13}+Q_0T_{14}\\
        P_0\\
        P_0T_{32}+\Theta_0T_{33}+Q_0T_{34}\\
        P_0T_{42}+\Theta_0T_{43}+Q_0T_{44} 
    \end{pmatrix}.
\end{align}
With the boundary condition $Z_1(0)=(0,1,0,0)^T$,
we write the solution
\begin{align}
    \displaystyle Z_1(x)=(
        T_{12},\,1,\,T_{32},\,T_{42}
    )^T.
\end{align}
Similarly, with $Z_2(0)=(0,0,0,1)^T$,
we have
\begin{align}
\displaystyle Z_2(x)=(
    T_{14},\,0,\,T_{34},\,T_{44})^T.
\end{align}
Following the same arguments as above we can write the solutions $Z_3$ and $Z_4$ with the prescribed boundary conditions $Z_3(1)=(0,1,0,0)$ and $Z_4(1)=(0,0,0,1).$\\

{\textbf{Step 4:}} Direct calculations using values $Z_k(1)$ for $E(0,\lambda)$ then give
\begin{align}\label{evans0}\begin{split}
    E(0,\beta)=&\det(Z_1(1)\,|\,Z_2(1)\,|\,Z_3(1)\,|\,Z_4(1))=\det(\Phi(1)Z_1(0)|\Phi(1)Z_2(0)|e_2|e_4)\\
    =&\det\begin{pmatrix}
        T_{12}(1) & T_{14}(1) & 0 & 0 \\
        1 & 0 & 1 & 0\\
        T_{32}(1) & T_{34}(1) & 0 & 0\\
        T_{42}(1) & T_{44}(1) & 0 & 1
    \end{pmatrix}=-T_{12}(1)T_{34}(1)+T_{14}(1)T_{32}(1).
    \end{split}
\end{align}
Now, using the identities in \eqref{dpeta} and \eqref{dtheta}, we evaluate the terms above by differentiating the first and third equations in \eqref{nls} to obtain
\begin{align}
\begin{split}\label{Ts}
    \frac{\partial \theta}{\partial p_0}\big|_{x=1}&=\frac{\sqrt{\beta}}{c(\theta_0)\sqrt{2p_0}}+2\frac{g(\theta_0)}{h(\theta_0)p_0}\beta-\frac{2}{c(\theta_0)p_0}\beta^{3/2}M^g,\\
    \frac{\partial u}{\partial \eta_0}\big|_{x=1}&= \frac{2}{c(\theta_0)}\sqrt{\beta}M-\frac{2}{cg(\theta_0)}\sqrt{\beta}M^g,\\
    \frac{\partial \theta}{\partial \eta_0}\big|_{x=1}&= 2\sqrt{\frac{2}{p_0}}\frac{g(\theta_0)}{c(\theta_0)h(\theta_0)}\sqrt{\beta}-2\sqrt{\frac{2}{p_0}}\frac{1}{c^2(\theta_0)}\beta M^g,
    \end{split}
\end{align}
\vspace{-1em}
    \begin{align*}
      \frac{\partial u}{\partial p_0}\big|_{x=1}&=\frac{1}{2g(\theta_0)}+\frac{1}{\sqrt{2p_0}}S(\tilde{\theta},\theta_0)+\sqrt{\frac{2}{p_0}}\beta M-\sqrt{\frac{2}{p_0}}\frac{1}{g(\theta_0)}\beta M^g.
    \end{align*}
Recall  \eqref{TT} and substitute \eqref{Ts} into \eqref{evans0} to  get
\begin{align}\label{evans1}
    E(0,\beta)=\frac{-2\beta}{p_0c^2(\theta_0)}\bigg[\Big(\frac{c(\theta_0)g(\theta_0)}{\sqrt{\beta}h(\theta_0)}-M^g\Big)S(\tilde{\theta},\theta_0)+\sqrt{\frac{p_0}{2}}\frac{c(\theta_0)}{\sqrt{\beta}h(\theta_0)}-\sqrt{\frac{p_0}{2}}M\bigg]
\end{align}
where
\begin{align*}
    M^g=\int_{\tilde{\theta}}^{\theta_0}(\frac{cg}{h})'(z)\frac{1}{\sqrt{G(z)-G(\tilde{\theta})}}\,dz,\quad
    M=\int_{\tilde{\theta}}^{\theta_0}(\frac{c}{h})'(z)\frac{1}{\sqrt{G(z)-G(\tilde{\theta})}}\,dz,
\end{align*}
and recalling that
\begin{align*}
    \sqrt{\frac{p_0}{2}}&=\int_{\tilde{\theta}}^{\theta_0}\frac{c(z)}{\sqrt{G(z)-G(\tilde{\theta})}}\,dz,\quad
S(\tilde{\theta},\theta_0):=\int_{\tilde{\theta}}^{\theta_0}\frac{c(z)}{g(z)\sqrt{G(z)-G(\tilde{\theta})}}\,dz.
\end{align*}
Using the change of variables that was introduced for $\bar u$ in \eqref{ubar}, we can write
\begin{align}
    \begin{split}
      M^g&=\int_{0}^{\beta}(\frac{cg}{h})'(F(t-\beta))\frac{F'(t-\beta)}{\sqrt{t}}\,dt,\\
    M&=\int_{0}^{\beta}(\frac{c}{h})'(F(t-\beta))\frac{F'(t-\beta)}{\sqrt{t}}\,dt,\\
    \sqrt{\frac{p_0}{2}}&=\int_{0}^{\beta}\frac{cg}{h}(F(t-\beta))\frac{1}{\sqrt{t}}\,dt,\quad
S(\tilde{\theta},\theta_0)=\int_0^\beta\frac{c}{h}(F(t-\beta))\frac{1}{\sqrt{t}}\,dt.
    \end{split}
\end{align}
Now, we compute $D'(\beta)$ directly using the expression given by \eqref{ubar}:
\begin{align*}
    D(\beta)=\int_0^\beta\frac{cg}{h}(F(t-\beta))\frac{1}{\sqrt{t}}\,dt\int_0^\beta\frac{c}{h}(F(t-\beta))\frac{1}{\sqrt{t}}\,dt.
\end{align*}
We obtain
\begin{align*}
    D'(\beta)=&\frac{cg}{h}(\theta_0)\frac{1}{\sqrt{\beta}}\int_0^\beta\frac{c}{h}(F(t-\beta))\frac{1}{\sqrt{t}}\,dt\\
    &-\int_0^\beta(\frac{cg}{h})'\frac{F'(t-\beta)}{\sqrt{t}}\,dt\int_0^\beta\frac{c}{h}(F(t-\beta))\frac{1}{\sqrt{t}}\,dt\\
    &+\frac{c(\theta_0)}{h(\theta_0)}\frac{1}{\sqrt{\beta}}\int_0^\beta\frac{cg}{h}(F(t-\beta))\frac{1}{\sqrt{t}}\,dt\\
    &-\int_0^\beta\frac{cg}{h}(F(t-\beta))\frac{1}{\sqrt{t}}\,dt\int_0^\beta(\frac{c}{h})'\frac{F'(t-\beta)}{\sqrt{t}}\,dt.
\end{align*}
This expression, up to a factor, matches exactly the expression of $E$ in \eqref{evans1} that gives (\ref{ed'}). 
\qed

A natural question rises immediately: What is the stability/instability of the bifurcating solutions? In the next section we provide a partial answer to this question.

\section{Spectral stability/instability of stationary solutions}\label{SpecStab}
\setcounter{equation}{0}
The goal of this section is to find a criterion for instability of the bifurcating stationary solutions. 
The main result of this section, which will be stated at the end, follows from the explicit computations. 

\subsection{A motivation for the bifurcation of the zero eigenvalue}
First, we recall the setup.  For any solution $\beta>0$ of $\bar u=2D(\beta)$, there is a unique stationary solution (Theorem \ref{propD}). The eigenvalue problem associated to the linearization around the stationary solution yields the following first order ODE:
\begin{align}\label{odee}
    Z'=(A(x,\beta)+\lambda B(x,\beta))Z
\end{align}
and we constructed the Evans function as follows,
\begin{align*}
    E(\lambda,\beta)=\det (Z_1|Z_2|Z_3|Z_4)
\end{align*}
where the $Z_i(x;\lambda,\beta)$ are independent solutions \eqref{odee} corresponding to different boundary conditions at $x=0$ and $x=1.$ We are interested in those solutions but only with $\lambda=0$ as an eigenvalue of the linear system. To motivate what we do next, we consider the curve $(\lambda(\beta),\beta)$ over which we  have $\lambda(\beta^*)=0$ and 
\begin{align*}
    E(\lambda(\beta),\beta)=0.
\end{align*}
Differentiate with respect to $\beta$ to get,
\begin{align*}
    E_\beta(\lambda(\beta),\beta)+E_\lambda(\lambda(\beta),\beta)\lambda'(\beta)=0,
\end{align*}
solving for $\lambda'$ at $\beta^*,$
\begin{align}\label{bifur0}
    \lambda'(\beta^*)=-\frac{E_\beta(\lambda(\beta^*),\beta^*)}{E_\lambda(\lambda(\beta^*),\beta^*)}.
\end{align}
Recall from \eqref{ed'} that
\begin{align*}
    E(0,\beta)=\frac{-2\beta}{p_0c^2(\theta_0)}D'(\beta).
\end{align*}
Differentiate the above with respect to $\beta$ to get
\begin{align*}
    E_\beta(0,\beta)=\frac{-2}{p_0c^2(\theta_0)}D'(\beta)-\frac{2\beta}{p_0c^2(\theta_0)}D''(\beta).
\end{align*}
At the bifurcation moment $\beta=\beta^*$ where $E(0,\beta^*)=0$, we have $D'(\beta^*)=0.$ Hence,
\begin{align}\label{Ebeta}
    E_\beta(0,\beta^*)=-\frac{2\beta}{p_0c^2(\theta_0)}D''(\beta^*).
\end{align}
If  $D''(\beta^*)\neq 0$ and $E_\lambda(0,\beta^*)\neq 0$, then the $\lambda(\beta)$ passes through $0$ transversally as $\beta$ passes through $\beta^*$. 
Therefore, it remains to examine $E_\lambda(\lambda(\beta),\beta)$.

\subsection{Computation of $E_\lambda(0,\beta^*)$}
First, we directly differentiate $E(\lambda,\beta)$ with respect to $\lambda,$
\begin{align*}
E_\lambda=&\det(Z_{1,\lambda}|Z_2|Z_3|Z_4)+\det(Z_1|Z_{2,\lambda}|Z_3|Z_4)\\
&+\det(Z_1|Z_2|Z_{3,\lambda}|Z_4)+\det(Z_1|Z_2|Z_{3}|Z_{4,\lambda}).
\end{align*}
To determine $Z_{i,\lambda}$, we differentiate the ODE \eqref{odee} for $Z_i$  with respect to $\lambda$ to get
\begin{align*}
    Z'_{i,\lambda}(x;\lambda)=(A(x,\beta)+\lambda B(x,\beta))Z_{i,\lambda}(x;\lambda)+B(x,\beta) Z_i(x;\lambda).
\end{align*}
Denote $W_i(x;\lambda)\equiv Z_{i,\lambda}(x;\lambda).$ The equation to be solved is
\begin{align*}
    W_i'(x;\lambda)=(A(x,\beta)+\lambda B(x,\beta))W_i+B(x;\beta) Z_i(x;\lambda).
\end{align*}
By the variation of parameters method:
\begin{align*}
    W_i(x;\lambda)\equiv Z_{i,\lambda}(x;\lambda)=\Phi(x;\lambda)\int_0^x\Phi^{-1}(s;\lambda)B(s;\lambda) Z_i(s;\lambda)\,ds.
\end{align*}
Recall,
\begin{align}
  \Phi(x;\lambda):=D\phi^x(z_0)= \begin{pmatrix}
      1 & T_{12} & T_{13} & T_{14}\\
      0 & 1 & 0 & 0\\
      0 & T_{32} & T_{33} & T_{34}\\
      0 & T_{42} & T_{43} & T_{44}
  \end{pmatrix}
\end{align}
where
\begin{align}\begin{split}
    T_{12}:=\frac{\partial u}{\partial p_0} \quad T_{13}:=\frac{\partial u}{\partial \theta_0} \quad T_{14}:=\frac{\partial u}{\partial \eta_0},\\
    T_{32}:=\frac{\partial \theta}{\partial p_0} \quad T_{33}:=\frac{\partial \theta}{\partial \theta_0} \quad T_{34}:=\frac{\partial \theta}{\partial \eta_0},\\
    T_{42}:=\frac{\partial \eta}{\partial p_0} \quad T_{43}:=\frac{\partial \eta}{\partial \theta_0} \quad T_{44}:=\frac{\partial \eta}{\partial \eta_0}.
    \end{split}
\end{align}
Since $\mathrm{tr}(A(x)+\lambda B(x))=0$, $\mathrm{det}\Phi(x;\lambda)=T_{33}T_{44}-T_{34}T_{43}=1$. Therefore,

\begin{align*}
        \Phi^{-1}=\begin{pmatrix}
            1 & -J& T_{14}T_{43}-T_{13}T_{44} & T_{13}T_{34}-T_{14}T_{33}\\
            0 & 1 & 0 & 0\\
            0 & T_{34}T_{42}-T_{32}T_{44} & T_{44} & -T_{34}\\
            0 & T_{32}T_{43}-T_{33}T_{42} & -T_{43} & T_{33}
        \end{pmatrix}
\end{align*}
where 
\begin{align*}
    J:=\det\begin{pmatrix}
        T_{12} & T_{13} & T_{14}\\
        T_{32} & T_{33} & T_{34}\\
        T_{42} & T_{43} & T_{44}
    \end{pmatrix}.
\end{align*}
Recall,
\begin{align*}
    B(x,\beta^*)=\begin{pmatrix}
0 & 0 & -\frac{h(\theta)}{g(\theta)} & 0\\
1 & 0 & 0 & 0\\
0 & 0 & 0 & 0\\
0 & 0 & \gamma_1-\frac{h^2(\theta)}{g(\theta)} & 0
\end{pmatrix}.
\end{align*}
Thus, $K:=\Phi^{-1}(x)B(x,\beta^*)$ is given by
\begin{align*}
   K= \begin{pmatrix}
        -J & 0 & -\frac{h}{g}+(\gamma_1-\frac{h^2}{g})(T_{13}T_{34}-T_{14}T_{33}) & 0\\
        1 & 0 & 0 & 0\\
        T_{34}T_{42}-T_{32}T_{44} & 0 & -(\gamma_1-\frac{h^2}{g})T_{34} & 0\\
        T_{32}T_{43}-T_{33}T_{42} & 0 & (\gamma_1-\frac{h^2}{g})T_{33} & 0
    \end{pmatrix}
\end{align*}
Altogether, at $x=1$ and for $\lambda=0$,
\begin{align*}\displaystyle
    Z_{1,\lambda}(1;0)=\Phi(1)
    \bigintss_0^1\begin{pmatrix}
        -T_{12}J+T_{32}K_{13}\\
        T_{12}\\
        T_{12}K_{31}+T_{32}K_{33}\\
        T_{12}K_{41}+T_{32}K_{43}
    \end{pmatrix}(s)ds
\end{align*}
\begin{align*}\displaystyle
    Z_{2,\lambda}(1;0)=\Phi(1)\bigintss_0^1\begin{pmatrix}
        -T_{14}J+T_{34}K_{13}\\
        T_{14}\\
        T_{14}K_{31}+T_{34}K_{33}\\
        T_{14}K_{41}+T_{34}K_{43}
    \end{pmatrix}(s)ds
\end{align*}
Recall and observe that at $x=1, \lambda=0$,
\begin{align*}
E_{\lambda}(1,0,\beta)
=&\det(Z_{1,\lambda}|Z_2|e_2|e_4)+\det(Z_1|Z_{2,\lambda}|e_2|e_4)\\
&+\det(Z_1|Z_2|{\vec 0}|Z_4)
+\det(Z_1|Z_2|Z_{3}|{\vec 0}).
\end{align*}
Direct calculations yield,
\begin{align}\label{E_l}
    E_\lambda=\underbrace{-(Z_{1,\lambda})^1T_{34}(1)+(Z_{1,\lambda})^3T_{14}(1)}_{(I)}\,\,\,\underbrace{-(Z_{2,\lambda})^3T_{12}(1)+(Z_{2,\lambda})^1T_{32}(1)}_{(II)}.
\end{align}
where $(Z_{i,\lambda})^{j}$ is the $jth$ component of the vector $Z_{i,\lambda}.$ Calculating the terms in \eqref{E_l},
\begin{itemize}
    \item The term $(I)$ in \eqref{E_l}:
\begin{align*}
    (Z_{1,\lambda})^1=&\int_0^1-T_{12}(s)J(s)+T_{32}(s)K_{13}(s)+T_{12}(1)T_{12}(s)\\
    &+{T_{13}(1)\big[T_{12}(s)K_{31}+T_{32}(s)K_{33}\big]\,}\\
    &+T_{14}(1)\big[T_{12}(s)K_{41}+T_{32}(s)K_{43}\big]ds\\
    =&\int_0^1-T_{12}(s)J(s)+T_{32}(s)K_{13}(s)+T_{12}(1)T_{12}(s)\\
    &+T_{14}(1)\big[T_{12}(s)K_{41}+T_{32}(s)K_{43}\big]ds\\
\end{align*}
Note: The second equality follows from $T_{13}(1)=0.$ Recall, $T_{13}(x):=\frac{\partial}{\partial \theta_0}u(x;u_0,p_0,\theta_0,\eta_0).$ So, at $x=1,$ we have $T_{13}(1)=\frac{\partial u_0}{\partial \theta_0}=0$ since $u_0$ and $\theta_0$ are the boundary conditions for the original unknowns.
\begin{align*}
    (Z_{1,\lambda})^3=&\int_0^1T_{32}(1)T_{12}(s)+{{T_{33}(1)}\,}\big[T_{12}(s)K_{31}+T_{32}(s)K_{33}\big]\\
    &+T_{34}(1)\big[T_{12}(s)K_{41}+T_{32}(s)K_{43}\big]ds\\
    =&\int_0^1T_{32}(1)T_{12}(s)+T_{12}(s)K_{31}+T_{32}(s)K_{33}\\
    &+T_{34}(1)\big[T_{12}(s)K_{41}+T_{32}(s)K_{43}\big]ds.
\end{align*}
Note that the second equality follows from 
   $T_{33}(1):=\frac{\partial\theta(1;\theta_0,p_0,\eta_0)}{\partial\theta_0}=\frac{\partial\theta_0}{\partial\theta_0}=1.$

Substituting into (I), one has
\begin{align}\nonumber
    I:=&-(Z_{1,\lambda})^1T_{34}(1)+(Z_{1,\lambda})^3T_{14}(1)\\\nonumber
    =&\int_0^1
    {-T_{34}(1)}\bigg(-T_{12}(s)J(s)+T_{32}(s)K_{13}(s)+T_{12}(1)T_{12}(s)\bigg)\\\nonumber
    &-T_{34}(1)T_{14}(1)\big[T_{12}(s)K_{41}+T_{32}(s)K_{43}\big]ds\\\nonumber
    &+T_{14}(1)\bigg(T_{32}(1)T_{12}(s)+T_{12}(s)K_{31}+T_{32}(s)K_{33}\bigg)\\\nonumber
    &+T_{14}(1)T_{34}(1)\big[T_{12}(s)K_{41}+T_{32}(s)K_{43}\big]\,ds\\\nonumber
    =&\int_0^1
    {-T_{34}(1)}\bigg(-T_{12}(s)J(s)+T_{32}(s)K_{13}\bigg)+{E(0,\beta^*)}T_{12}(s)\\\nonumber
    &+T_{14}(1)\bigg(T_{12}(s)K_{31}+T_{32}(s)K_{33}\bigg)ds\\\label{I}
    =&\int_0^1
    {-T_{34}(1)}\bigg(-T_{12}(s)J(s)+T_{32}(s)K_{13}\bigg)\\
    &+T_{14}(1)\bigg(T_{12}(s)K_{31}+T_{32}(s)K_{33}\bigg)ds.\nonumber
\end{align}
In the above, we have used  $E(0,\beta^*)=0$ since $\lambda(\beta^*)=0$ is an eigenvalue. 
\item The Term $(II)$ in \eqref{E_l}:
Similarly,
\begin{align*}
    (Z_{2,\lambda})^1=&\int_0^1-T_{14}(s)J(s)+T_{34}(s)K_{13}(s)+T_{12}(1)T_{14}(s)\\
    &\qquad +{{T_{13}(1)\big[T_{14}(s)K_{31}+T_{34}(s)K_{33}\big]\,}}\\
    &\qquad +T_{14}(1)\big[T_{14}(s)K_{41}+T_{34}(s)K_{43}\big]ds\\
    =&\int_0^1-T_{14}(s)J(s)+T_{34}(s)K_{13}(s)+T_{12}(1)T_{14}(s)\\
    &\qquad +T_{14}(1)\big[T_{14}(s)K_{41}+T_{34}(s)K_{43}\big]ds.
    \end{align*}
    \begin{align*}
    (Z_{2,\lambda})^3=&\int_0^1T_{32}(1)T_{14}(s)+{T_{33}(1)}\big[T_{14}(s)K_{31}+T_{34}(s)K_{33}\big]\\
    &\qquad +T_{34}(1)\big[T_{14}(s)K_{41}+T_{34}(s)K_{43}\big]ds\\
    =&\int_0^1T_{32}(1)T_{14}(s)+T_{14}(s)K_{31}+T_{34}(s)K_{33}\\
    &\qquad +T_{34}(1)\big[T_{14}(s)K_{41}+T_{34}(s)K_{43}\big]ds
\end{align*}
Substituting into $(II)$, one has
\begin{align}\nonumber
    -&(Z_{2,\lambda})^3T_{12}(1)+(Z_{2,\lambda})^1T_{32}(1)\\\nonumber
    =&\int_0^1{-T_{12}(1)}\bigg(T_{14}(s)K_{31}+T_{34}(s)K_{33}
    +T_{34}(1)\big[T_{14}(s)K_{41}+T_{34}(s)K_{43}\big]\bigg)\\\nonumber
    &+{T_{32}(1)}\bigg(T_{34}(s)K_{13}(s)-T_{14}(s)J(s)+T_{14}(1)\big[T_{14}(s)K_{41}+T_{34}(s)K_{43}\big]\bigg)ds\\\nonumber
    =&\int_0^1{-T_{12}(1)}\bigg(T_{14}(s)K_{31}+T_{34}(s)K_{33}\bigg)
    -T_{12}(1)T_{34}(1)\big[T_{14}(s)K_{41}+T_{34}(s)K_{43}\big]\\\nonumber
    &+{T_{32}(1)}\bigg(T_{34}(s)K_{13}(s)-T_{14}(s)J(s)\bigg)+T_{32}(1)T_{14}(1)\big[T_{14}(s)K_{41}+T_{34}(s)K_{43}\big]ds\\\nonumber
    =&\int_0^1{-T_{12}(1)}\bigg(T_{14}(s)K_{31}+T_{34}(s)K_{33}\bigg)
    +E(0,\beta^*)\big[T_{14}(s)K_{41}+T_{34}(s)K_{43}\big]\\\nonumber
    &+{T_{32}(1)}\bigg(T_{34}(s)K_{13}-T_{14}(s)J(s)\bigg)\,ds\\\label{II}
    =&\int_0^1{-T_{12}(1)}\bigg(T_{14}(s)K_{31}+T_{34}(s)K_{33}\bigg)
    +{T_{32}(1)}\bigg(T_{34}(s)K_{13}-T_{14}(s)J(s)\bigg)\,ds. 
\end{align}
\end{itemize}
Again we have used $E(0,\beta^*)=0$ since $\lambda(\beta^*)=0$ is an eigenvalue.
Finally, substituting \eqref{I} and \eqref{II} to \eqref{E_l}, we obtain
\begin{align}\label{numbere}\begin{split}
   E_\lambda(0,\beta^*)=&-(Z_{1,\lambda})^1T_{34}(1)+(Z_{1,\lambda})^3T_{14}(1)-(Z_{2,\lambda})^3T_{12}(1)+(Z_{2,\lambda})^1T_{32}(1)\\
    =&{T_{34}(1)}\int_0^1T_{12}(s)J(s)-T_{32}(s)K_{13}(s)\,ds\\
    &\quad +T_{14}(1)\int_0^1T_{12}(s)K_{31}(s)+T_{32}(s)K_{33}(s)\,ds\\
    &\quad +T_{12}(1)\int_0^1-T_{14}(s)K_{31}(s)
    -T_{34}(s)K_{33}(s)\,ds\\
    &\quad +{T_{32}(1)}\int_0^1 -T_{14}(s)J(s)+T_{34}(s)K_{13}(s)\,ds.
    \end{split}
\end{align}

We are now in a position to draw a conclusion based on the above computations.
{
\subsection{Bifurcation of zero eigenvalue}
Combining \eqref{bifur0}, \eqref{Ebeta} and \eqref{numbere}, we have, at $\beta^*$ with $\lambda(\beta^*)=0$, if $E_\lambda(0,\beta^*)\neq 0$, then 
\begin{align}\label{Dlambda}
 \lambda'(\beta^*)=\frac{2\beta^* }{p_0c^2(\theta_0)E_\lambda(0,\beta^*)}D''(\beta^*).
\end{align}

\begin{thm}\label{bif0ev} For the case $\gamma_1=|\gamma_2|,$ let $\beta^*$ be a local minimum of $D(\beta)$ (so that zero is an eigenvalue for the corresponding stationary solution). {Assume $E_\lambda(0,\beta^*)\neq 0$.}

 If $D''(\beta^*)\neq 0$ (so that $D''(\beta^*)> 0$), then there exists a real-valued function $\lambda(\beta)$ for $\beta\in (\beta^*-\delta,\beta^*+\delta)$ for some $\delta>0$ such that $E(\lambda(\beta),\beta)=0$ and the function $\lambda(\beta)$ satisfies $\lambda(\beta^*)=0$, $\lambda(\beta)$ takes opposite signs for $\beta$ on different sides of $\beta^*$. 
\end{thm}

We conclude this section with some important comments.
  \begin{itemize}
  \item[(i)]  We emphasize that Theorem \ref{e=d'} provides a necessary and sufficient condition for the existence of zero eigenvalue of a stationary solution and Theorem \ref{bif0ev} only describes behavior of the zero eigenvalue in the process of the saddle-node bifurcation. While one can conclude that the bifurcated stationary solution with a {\em positive} eigenvalue (bifurcated from the zero eigenvalue) is spectrally  unstable, we could not claim that the other bifurcated stationary solution with a negative eigenvalue (bifurcated from the zero eigenvalue) is spectrally stable -- since we don't have additional information on the spectrum, even for the stationary solution associated to $\beta^*$. There are evidences from numerical  simulations of the spectrum and of nonlinear evolution of \eqref{intro} and \eqref{ibv} that, for  the stationary solution associated to $\beta^*$, the spectrum except the zero eigenvalue lies in the left half plane. If this is the case, then the above criterion for behavior of zero eigenvalue near the bifurcation is truly a criterion for spectral stability of stationary solutions.
    
  \item[(ii)] We could not prove  but believe that generically $E_\lambda(0,\beta^*)<0$, which implies $\lambda'(\beta^*)<0$, and hence, $D'(\beta_+)>0$ and $D'(\beta_-)<0$.  This belief is supported by the linear stability proved in Section \ref{SmallStab} for stationary solutions with {\em small} $\bar{u}$ that associated to a portion with $D'(\beta)<0$.

\item[(iii)] Everything we have done regarding the multiplicity and stability of stationary solutions is for the case $\gamma_1=|\gamma_2|.$ One can observe that the case $\gamma_1>|\gamma_2|$ can be analyzed in a similar way.  The main difference is that $h(\theta)$ is   always positive since $\gamma_1>0$, which leads to a smooth map $D(\beta)$ defined for $\beta\in\R_+$ and the map $D$ would have countably many local maximum points (evolved from poles for $\gamma_1=|\gamma_2|$) alternated between each local minimum if $\gamma_1-|\gamma_2|$ is small {\em but} might be monotone if $\gamma_1-|\gamma_2|$ is large.
 In particular, the formulas \eqref{ed'} and \eqref{Dlambda} should hold true for the case $\gamma_1>|\gamma_2|$.
 \end{itemize}
}

\section{Linear stability for steady states with small $\bar{u}$}\label{SmallStab}
\setcounter{equation}{0}

Consider the stationary system
\begin{align}
\begin{split}\label{ss*}
(g(\theta)u_x)_x&=0,\\
    c(\theta)(c(\theta)\theta_x)_x&-h(\theta)u_x=0,
    \end{split}
\end{align}
equipped with the boundary conditions
\begin{align}
    \begin{split}\label{bc*}
    u(0)=0,&\quad u(1)=\bar u>0;\quad
        \theta(0)=\theta(1)=\theta_0.
    \end{split}
\end{align}
Recall the linear system \eqref{Upert} and \eqref{phipert},
\begin{align}\label{linearU}
U_t&=\big[g(\theta)U_x\big]_x+\big[g'(\theta)\Theta u_x\big]_x+\big[h(\theta)\Theta_t\big]_x,\\\label{linearth}
\gamma_1\Theta_t&=c'(\theta)(c(\theta)\theta_x)_x\Theta+c(\theta)(c(\theta)\Theta)_{xx}-h'(\theta)\Theta u_x-h(\theta)U_x,
\end{align}
with $U(t,0)=U(t,1)=\Theta(t,0)=\Theta(t,1)=0.$

We now define the weighted energy $E_b(t)$ as
\begin{align}
E_b(t):=\int_0^1U^2+&b\gamma_1\Theta^2+c(\theta)\Theta_x^2\,dx
\end{align}
for a constant $b>0.$

In this section we use the following notation, we say $A\lesssim B$ to mean $A\leq C B$ for some constant $C$ depending only on the bounds of the functions $c,g$ and $h$ and possibly their derivatives $c',g'$ and $h'$ which are still bounded. Denote the bounds for $g,h,$ and $c$ as follows
\begin{align}
    g_L&\leq g\leq g_U;\quad h_L\leq h\leq h_U;\quad c_L\leq c\leq c_U.
\end{align}
\begin{lem}\label{lemmasmallubar}
    Let $(u,\theta)$ be a $C^2$ solution to \eqref{ss*}-\eqref{bc*}. Then for any $0<\varepsilon<1,$ if $0<\bar u\leq\varepsilon$ we have 
\begin{align*}
 \max\{\|u_x\|_{L^\infty}, \|{\theta_x}\|_{L^\infty}\}\lesssim \varepsilon.
\end{align*}
Moreover, we have
\begin{align}
    \|\theta_{xx}\|_{L^\infty}\lesssim \varepsilon.
\end{align}
\end{lem}
\begin{proof}
    We start with $u_x,$
    The proof follows directly from the equations \eqref{ss*}. Integrating the $u$ equation and recall $g(\theta)u_x=p,$
    \begin{align*}
        \bar u=p\int_0^1\frac{1}{g(\theta(s))}\,ds.
    \end{align*}
    If $\bar u\leq\varepsilon$ then $p\leq g_U\varepsilon,$ where $g_U=\sup_{z\in\mathbb{R}}g(z).$
    
    Hence, using $u_x=\frac{p}{g(\theta)},$ \[\|u_x\|_{L^\infty}\lesssim\varepsilon.\]

    Now, for $\theta_x$ recall $\displaystyle G(\theta)=\int_{\theta_0}^\theta\frac{h(z)}{g(z)}\,dz.$ Direct estimate,
    \begin{align*}
        \sqrt{G(\theta)-G(\tilde\theta)}=\displaystyle\sqrt{\int_{\tilde\theta}^{\theta}\frac{h(z)}{g(z)}\,dz}\leq \sqrt{\frac{h_U}{g_L}}\sqrt{\theta-\tilde\theta}.
    \end{align*}
    
    By \eqref{Eq4SS}
    \begin{align*}
        \varepsilon\geq \bar u&=2\int_{\tilde\theta}^{{\theta_0}}\frac{c(\theta)}{\sqrt{G(\theta)-G(\tilde\theta)}}\,d\theta\,\int_{\tilde\theta}^{{\theta_0}}\frac{c(\theta)}{g(\theta)\sqrt{G(\theta)-G(\tilde\theta)}}\,d\theta\gtrsim\theta_0-\tilde\theta.
    \end{align*}
    Using \eqref{eq4theta} and the above estimates
    \begin{align*}
    \theta_x=\frac{\sqrt{2p}}{c(\theta)}\sqrt{G(\theta)-G(\tilde\theta)}\lesssim\sqrt{\varepsilon}\sqrt{\theta-\tilde\theta}\leq \sqrt{\varepsilon}\,\sqrt{\theta_0-\tilde\theta}\lesssim\varepsilon.
    \end{align*}
    Hence, \[\|\theta_x\|_{L^\infty}\lesssim\varepsilon.\]

    For the last assertion, we use the $\theta$ equation in \eqref{ss*} along with $\|u_x\|_{L^\infty}\lesssim\varepsilon$ and $\|\theta_x\|_{L^\infty}\lesssim\varepsilon,$ to get 
    \[\|\theta_{xx}\|_{L^\infty}\lesssim\varepsilon.\]
     The proof is complete.
\end{proof}
\begin{thm}\label{Stab}
    Let $(\theta,u)$ be a $C^2$ solution to \eqref{ss*} and \eqref{bc*} with $0<\bar u<\varepsilon$ for $\varepsilon$ sufficiently small.
    If the constant $b>0$ is sufficiently small, then the energy
    \begin{align}
E_b(t):=\int_0^1U^2+&b\gamma_1\Theta^2+c(\theta)\Theta_x^2\,dx
\end{align}
    associated with smooth solutions to the system \eqref{linearU} and \eqref{linearth} decays exponentially. Specifically, 
    \begin{align}
        E_b(t)<E_b(0)e^{-6Lt},
    \end{align}
    where $L>0$ depends on the constant $b,$ the norms $\|u_x\|_{L^\infty}, \|\theta_x\|_{L^\infty}, \|\theta_{xx}\|_{L^\infty}$ and the bounds of the given functions $c(\theta),g(\theta)$ and $h(\theta).$
\end{thm}
\begin{proof}
Let $(\theta,u)$ be a solution to the system \eqref{ss*} and \eqref{bc*}. Let $b>0$ be a constant. Multiply \eqref{linearU} by $U$ and integrate by parts,
\begin{align}\label{Uint}
    \frac{d}{dt}\frac{1}{2}\int_0^1U^2\,dx+\int_0^1g(\theta)U_x^2\,dx=\int_0^1-g'(\theta)u_x\Theta U_x-h(\theta)\Theta_tU_x\,dx
\end{align}
{Let $0<\sigma(x)<1$ be any function (to be chosen later)}. Apply Cauchy inequality to \eqref{Uint} to get,
\begin{align*}
    \frac{d}{dt}\frac{1}{2}&\int_0^1U^2\,dx+ \int_0^1(1-\sigma)g(\theta)U_x^2\,dx\\
    &= \int_0^1-\sigma\,g(\theta) U_x^2\,dx+\int_0^1-g'(\theta)u_x\Theta U_x-h(\theta)\Theta_tU_x\,dx\\ 
    &\leq  \int_0^1-\sigma\, g(\theta)U_x^2\,dx+\frac{1}{2}\int_0^1g'^2|u_x|U_x^2\,dx\\
    &\qquad +\frac{1}{2}\int_0^1|u_x|\Theta^2\,dx-\int_0^1h(\theta)\Theta_tU_x\,dx\\ 
    &\leq \int_0^1\big[-\sigma\, g(\theta)+\frac{1}{2}\|g'^2u_x\|_{L^\infty}\big]U_x^2\,dx+\frac{\|u_x\|_{L^\infty}}{2}\int_0^1\Theta^2\,dx\\
    &\qquad-\int_0^1h(\theta)\Theta_tU_x\,dx.
\end{align*}
By the assumption $\bar u<\varepsilon$ and Lemma \ref{lemmasmallubar} we have
\begin{align}\nonumber\label{uest}
  \frac{1}{2}\frac{d}{dt}\int_0^1U^2\,dx+ \int_0^1(1-\sigma)g(\theta)U_x^2\,dx
  \lesssim& \int_0^1\big[\varepsilon-\sigma g(\theta)\big]U_x^2\,dx\\
  &+ \varepsilon\int_0^1\Theta^2\,dx-\int_0^1h(\theta)\Theta_tU_x\,dx.
\end{align}
Now, multiply \eqref{linearth} by $b\Theta$ and integrate,
\begin{align*}
    \frac{d}{dt}\frac{1}{2}\gamma_1b\int_0^1\Theta^2\,dx=&\int_0^1 bc'(\theta)(c(\theta)\theta_x)_x\Theta^2+bc(\theta)(c(\theta)\Theta)_{xx}\Theta\\
    &-bh'(\theta) u_x\Theta^2+b(h(\theta)\Theta)_xU\,dx.
    \end{align*}
    Integrate by parts the second term in the right hand side to get
    \begin{align}\label{phiint}\begin{split}
    \frac{d}{dt}\frac{1}{2}\gamma_1b\int_0^1\Theta^2\,dx+\int_0^1bc^2(\theta)\Theta^2_{x}\,dx
    =& \int_0^1bc(\theta)c'(\theta)\theta_{xx}\Theta^2 -b(c^2)'\theta_x\Theta\Theta_x\\
    &-bh'(\theta) u_x\Theta^2+b(h(\theta)\Theta)_xU\,dx.
    \end{split}
\end{align}
Using Poinca\'e inequality for the second term in the left hand side,
\begin{align*}
\int_0^1bc^2\Theta_x^2\,dx&\geq\frac{1}{2}bc^2_L\int_0^1\Theta_x^2\,dx+\frac{1}{2}bc^2_L\int_0^1\Theta^2_x\,dx\\
    &\geq \frac{1}{2}bc^2_L\int_0^1\Theta_x^2\,dx+\frac{1}{8}bc^2_L\int_0^1\Theta^2\,dx
\end{align*}
\eqref{phiint} becomes,
\begin{align}\nonumber
    b\frac{\gamma_1}{2}\frac{d}{dt}&\int_0^1\Theta^2\,dx+\frac{1}{2}bc^2_L\int_0^1\Theta^2_{x}\,dx+\frac{1}{8}bc^2_L\int_0^1\Theta^2\,dx\\\nonumber
    \leq&
    \int_0^1 b\frac{1}{2}(c^2)'\theta_{xx}\Theta^2-b(c^2)'\theta_x\Theta\Theta_x -bh'(\theta) u_x\Theta^2+b(h(\theta)\Theta)_xU\,dx\\ \nonumber
    \leq& \frac{b}{2}\bigg[\|(c^2)'\theta_{xx}\|_{L^\infty}+2\|h'u_x\|_{L^\infty}+\|h'\theta_x\|^2_{L^\infty}+\|(c^2)'\|_{L^\infty}^2\|\theta_x\|^2_{L^\infty}\bigg]\int_0^1\Theta^2\,dx\\
    &+\frac{b}{2}\big(\|h\|^2_{L^\infty}+1\big)\int_0^1\Theta_x^2\,dx+b\int_0^1U^2\,dx.
\end{align}
By the assumption $\bar u<\varepsilon$ and Lemma \ref{lemmasmallubar} we have
\begin{align}\label{phiest}
    \frac{1}{2}\frac{d}{dt}\gamma_1b\int_0^1\Theta^2\,dx&+\frac{1}{2}bc^2_L\int_0^1\Theta^2_{x}\,dx+\frac{1}{8}bc^2_L\int_0^1\Theta^2\,dx\\\nonumber
    \lesssim &\,\varepsilon \,b\int_0^1\Theta^2\,dx+b\int_0^1U^2\,dx+\frac{b}{2}\int_0^1\Theta_x^2\,dx.
\end{align}
Multiplying \eqref{linearth} by $\Theta_t,$ we obtain
\begin{align}\label{phixint}\begin{split}
    \frac{1}{2}\frac{d}{dt}\int_0^1c^2(\theta)\Theta_x^2\,dx=&\int_0^1-\gamma_1\Theta_t^2+\big[c'(\theta)(c(\theta)\theta_x)_x+c(\theta)(c'(\theta)\theta_x)_x\big]\Theta\Theta_t\\
    &-h'(\theta)u_x\Theta\Theta_t-h(\theta)U_x\Theta_t\,dx
    \end{split}
\end{align}
By the assumption $\bar u<\varepsilon,$ Lemma \ref{lemmasmallubar} and Young inequality we have
\begin{align}\label{phixest}\begin{split}
    \frac{1}{2}\frac{d}{dt}\int_0^1c^2(\theta)\Theta_x^2\,dx\lesssim &\int_0^1\big[-\gamma_1+\varepsilon\big]\Theta_t^2\,dx\\
&+\varepsilon\int_0^1\Theta^2\,dx-\int_0^1h(\theta)U_x\Theta_t\,dx.
\end{split}
\end{align}

Summing the three inequalities \eqref{uest}, \eqref{phiest}, and \eqref{phixest},

\begin{align*}\nonumber
    \frac{1}{2}\frac{d}{dt}&\int_0^1U^2+\gamma_1b\Theta^2+c(\theta)\Theta^2_x\,dx\\
    &\qquad+\int_0^1(1-\sigma)g(\theta)U^2_x\,dx+\frac{1}{2}bc_L^2\int_0^1\Theta_x^2\,dx+\frac{1}{8}bc_L^2\int_0^1\Theta^2\,dx\\\nonumber
    &\lesssim\int_0^1[-\sigma g+\varepsilon]U_x^2-2h\Theta_tU_x+[-\gamma_1+\varepsilon]\Theta^2_t\,dx\\
    &\qquad+[\varepsilon b+2\varepsilon]\int_0^1\Theta^2\,dx+b\int_0^1U^2\,dx\\
    &=-\int_0^1\bigg(\sqrt{\gamma_1-\varepsilon}\Theta_t+\sqrt{\sigma g-\varepsilon}U_x\bigg)^2\,dx\\
    &\qquad+[\varepsilon b+2\varepsilon]\int_0^1\Theta^2\,dx+b\int_0^1U^2\,dx
\end{align*}
{where the last line follows by choosing $\sigma(\theta(x);\varepsilon):=\frac{1}{g(\theta(x))}(\varepsilon+\frac{h^2(\theta(x))}{\gamma_1-\varepsilon}).$}

Using the crucial inequality \eqref{positiveDamping} which reads $\frac{h^2}{\gamma_1 g}<1$ keeps $\sigma(x)<1$ provided $\varepsilon>0$ small enough. Now we have
\begin{align*}
   \frac{1}{2}\frac{d}{dt}&\int_0^1U^2+\gamma_1b\Theta^2+c(\theta)\Theta^2_x\,dx
   \lesssim  -\int_0^1(1-\sigma)g(\theta)U^2_x\,dx-\frac{1}{2}bc_L^2\int_0^1\Theta_x^2\,dx\\
   &\qquad\qquad -[-\varepsilon b-2\varepsilon+\frac{1}{8}bc_L^2]\int_0^1\Theta^2\,dx+b\int_0^1U^2\,dx\\
   &\qquad\leq  -\big[\frac{1}{2}\min\{(1-\sigma)g\}-b\big]\int_0^1U^2\,dx\\
   &\qquad\qquad-[-\varepsilon b-2\varepsilon+\frac{1}{8}bc_L^2]\int_0^1\Theta^2\,dx-\frac{1}{2}bc_L^2\int_0^1\Theta_x^2\,dx\\
  &\qquad \leq -\min\{r_1,r_2,r_3\}\int_0^1U^2+\gamma_1b\Theta^2+c(\theta)\Theta^2_x\,dx
\end{align*}
where
\begin{align*}
    r_1&:=\frac{1}{2}\min\{\big[1-\sigma(\theta(x);\varepsilon)\big]g(\theta(x))\}-b,\quad
    r_2:=\frac{-\varepsilon b-2\varepsilon+\frac{1}{8}bc_L^2}{\gamma_1b},\quad
    r_3:=\frac{1}{2}b\frac{c^2_L}{c_U}
\end{align*}
where $r_i>0$ for $i=1,2,3$ whenever $0<b<\frac{1}{2}\displaystyle\min_{0\leq x\leq1}\{\big[1-\sigma(\theta(x);\varepsilon)\big]g(\theta(x))\}$ and $\varepsilon>0$ sufficiently small.

{ Recall that $\sigma(\theta(x);\varepsilon)$ is chosen to be $\sigma(\theta(x);\varepsilon)=\frac{1}{g(\theta(x))}(\frac{h^2(\theta(x))}{\gamma_1-\varepsilon}+\varepsilon).$ Note that by \eqref{positiveDamping},  $$(1-\sigma)g\to g(\theta)-\displaystyle\frac{h^2(\theta)}{\gamma_1}\geq \bar C>0\quad \text{as\,\,\,} \varepsilon\to 0.$$ Therefore, for $\varepsilon>0$ arbitrarily small there is always $0<b<\frac{1}{2}\displaystyle\min_{0\leq x\leq1}\{(1-\sigma(\theta(x);\varepsilon))g(\theta(x))\}.$

    Moreover, again, by \eqref{positiveDamping} 
    \[\sigma(\theta(x);0)=\frac{h^2(\theta(x))}{\gamma_1g(\theta(x))}<1\] so by continuity $\sigma(\theta(x);\varepsilon)<1$ for $\varepsilon>0$ sufficiently small.
}
Hence, the exponential decay follows
 \[  E(t)< E(0)e^{-L\,t} \quad\text{where}\quad L:=\min\big\{r_1,r_2,r_3\big\}.\]
 
 This completes the proof.
\end{proof}

\noindent
{\bf Acknowledgment.}
    {The authors thank the referee for his/her careful review and helpful comments on our original submission}. This project was initiated while M. Sofiani was a Ph.D. student at the University of Kansas and continued as a postdoctoral researcher at King Abdullah University of Science and Technology (KAUST). The authors gratefully acknowledge the support of both institutions.\\
     
     \noindent
     {\bf Funding.}
     M. Sofiani acknowledges the support of King Abdullah University of Science and Technology (KAUST) for his postdoctoral research through Fund ID BAS/1/1652-01-01.
     
\bigskip

\phantomsection
\addcontentsline{toc}{section}{References}
\bibliographystyle{plain}
\bibliography{ref}

@article{BC,
author = {A. Bressan and G. Chen},
journal = {Ann. I. H. Poincar\'{e}--AN},
number = {no. 2},
pages = {335-354},
title = {Generic regularity of conservative solutions to a nonlinear wave equation},
volume = {34},
year = {2017},
}

@article{liucalderer00,
author = {M. Calderer and C. Liu},
journal = {SIAM J. Appl. Math},
pages = {1925-1949},
title = {Liquid crystal flow: Dynamic and static configurations},
volume = {60},
year = {2000},
}

@article{CHL20,
author = {G. Chen and T. Huang and W. Liu},
journal = {Arch. Ration. Mech. Anal},
number = {no. 2},
pages = {839-891},
title = {Poiseuille flow of nematic liquid crystals via the full {E}ricksen-{L}eslie model},
volume = {236},
year = {2020},
}

@article{CS22,
author = {G. Chen and M. Sofiani},
journal = {Communications on Applied Mathematics and Computation},
pages = {1-18},
title = {Singularity formation for the general Poiseuille flow of nematic liquid crystals},
year = {2022},
}

@book{DeGP,
address = {{\bf 83}, Oxford Science Publications},
author = {P. G. De Gennes and J. Prost},
edition = {2nd},
publisher = {International Series of Monographs on Physics},
title = {The Physics of Liquid Crystals},
year = {1995},
}

@incollection{Eri76,
address = {ed.), Vol. 2,. Academic Press, New York},
author = {J. L. Ericksen},
booktitle = {Advances in Liquid Crystals},
pages = {233-298},
publisher = {G. H. Brown},
title = {Equilibrium Theory of Liquid Crystals},
year = {1976},
}

@article{ericksen62,
author = {J. L. Ericksen},
journal = {Arch. Ration. Mech. Anal},
pages = {371-378},
title = {Hydrostatic theory of liquid crystals},
volume = {9},
year = {1962},
}

@article{frank58,
  title={I. Liquid crystals. On the theory of liquid crystals},
  author={Frank, F. C.},
  journal={Discussions of the Faraday Society},
  volume={25},
  pages={19--28},
  year={1958},
  publisher={Royal Society of Chemistry}
}

@article{leslie68,
author = {F. M. Leslie},
journal = {Proc. Roy. Soc. A},
pages = {359-372},
title = {Some thermal effects in cholesteric liquid crystals},
volume = {307},
year = {1968},
}

@incollection{Les,
address = {New York},
author = {F. M. Leslie},
booktitle = {{\em Advances in Liquid Crystals}, Vol. 4},
pages = {1-81},
publisher = {Academic Press},
title = {Theory of Flow Phenomena in Liquid Crystals},
year = {1979},
}

@article{lin89,
author = {F. H. Lin},
journal = {Comm. Pure Appl. Math},
pages = {789-814},
title = {Nonlinear theory of defects in nematic liquid crystals; phase transition and phenomena},
volume = {42},
year = {1989},
}

@article{oseen33,
author = {C. W. Oseen},
journal = {Trans. Faraday Soc},
number = {no. 140},
pages = {883-899},
title = {The theory of liquid crystals},
volume = {29},
year = {1933},
}

@article{JDEDL,
author = {T. Dorn and W. Liu},
journal = {J. Differential Equations},
number = {no. 12},
pages = {3184-3210},
title = {Steady-states for shear flows of a liquid-crystal model: Multiplicity, stability, and hysteresis},
volume = {253},
year = {2012},
}

@article{JDDEW,
author = {J. Jiao and K. Huang and W. Liu},
journal = {J. Dynam. Diff. Equations},
number = {no. 34},
pages = {239-269},
title = {Stationary shear flows of nematic liquid crystals: A comprehensive study via {E}ricksen-{L}eslie model},
volume = {253},
year = {2022},
}

@article{chen2024initial,
  title={Initial-Boundary Value Problems for Poiseuille Flow of Nematic Liquid Crystal via Full {E}ricksen--{L}eslie Model},
  author={G. Chen and Y. Hu and Q. Zhang},
  journal={SIAM Journal on Mathematical Analysis},
  volume={56},
  number={2},
  pages={1809--1850},
  year={2024},
  publisher={SIAM}
}

@article{leslie1968some,
  title={Some constitutive equations for liquid crystals},
  author={Leslie, F. M.},
  journal={Archive for Rational Mechanics and Analysis},
  volume={28},
  pages={265--283},
  year={1968},
  publisher={Springer}
}

@article{Zocher1927,
  title={Uber die Einwirkung magnetischer, elektrischer und mechanischer Kr{\"a}fte auf Mesophasen},
  author={Zocher, H.},
  journal={Physik. Zietschr},
  volume={28},
  pages={790--796},
  year={1927}
}

@article{ericksen1961,
  title={Conservation laws for liquid crystals},
  author={Ericksen, J. L},
  journal={Transactions of the Society of Rheology},
  volume={5},
  number={1},
  pages={23--34},
  year={1961},
  publisher={The Society of Rheology}
}

@article{CLS,
  title={The Poiseuille flow of the full Ericksen-Leslie model for nematic liquid crystals: The general case},
  author={G. Chen and W. Liu and M. Sofiani},
  journal={Journal of Differential Equations},
  volume={376},
  pages={538--573},
  year={2023},
  publisher={Elsevier}
}

@article{chen2024singularity,
  title={Singularity formation for full Ericksen--Leslie system of nematic liquid crystal flows in dimension two},
  author={G. Chen and T. Huang and X. Xu},
  journal={SIAM Journal on Mathematical Analysis},
  volume={56},
  number={3},
  pages={3968--4005},
  year={2024},
  publisher={SIAM}
}

@article{lin2014recent,
  title={Recent developments of analysis for hydrodynamic flow of nematic liquid crystals},
  author={F. H. Lin and C. Wang},
  journal={Philosophical Transactions of the Royal Society A: Mathematical, Physical and Engineering Sciences},
  volume={372},
  number={2029},
  pages={20130361},
  year={2014},
  publisher={The Royal Society Publishing}
}

@article{lin2016global,
  title={Global existence of weak solutions of the nematic liquid crystal flow in dimension three},
  author={F. H. Lin and C. Wang},
  journal={Communications on Pure and Applied Mathematics},
  volume={69},
  number={8},
  pages={1532--1571},
  year={2016},
  publisher={Wiley Online Library}
}

@article{jiang2022zero,
  title={The zero inertia limit from hyperbolic to parabolic Ericksen-Leslie system of liquid crystal flow},
  author={N. Jiang and Y.-L. Luo},
  journal={Journal of Functional Analysis},
  volume={282},
  number={1},
  pages={109280},
  year={2022},
  publisher={Elsevier}
}

\end{document}